\documentclass[a4paper,11pt]{article}

\usepackage[T1]{fontenc}
\usepackage{lmodern}
\usepackage{eurosym}
\usepackage{lastpage}
\usepackage{xspace}
\usepackage{fancyhdr}
\usepackage{booktabs}
\usepackage{graphicx}
\usepackage{multirow}
\usepackage{array}
\usepackage{xcolor}
\usepackage{csquotes}
\usepackage{titlesec}
\usepackage{hyperref}
\usepackage{comment}
\usepackage{amsthm}
\usepackage{amsmath}
\usepackage{amssymb}
\usepackage{pifont}
\usepackage{tabto}
\usepackage{enumitem}
\usepackage{makecell}
\usepackage{graphicx}

\titleformat*{\section}{\large\bfseries}
\titleformat*{\subsection}{\normalsize\bfseries}
\titleformat*{\subsubsection}{\normalsize\bfseries}

\renewcommand{\footnotesize}{\fontsize{8bp}{1em}\selectfont}

\newtheorem{theorem}{Theorem}

\newtheorem{proposition}[theorem]{Proposition}

\newtheorem{definition}[theorem]{Definition}
\newtheorem{example}[theorem]{Example}
\newtheorem{remark}[theorem]{Remark}
\newtheorem{corollary}[theorem]{Corollary}

\newtheorem{problem}[theorem]{Problem}
\usepackage{authblk}

\usepackage[margin=15mm,includehead,includefoot]{geometry}

\begin{document}

\title{\bf Detection of coordinated fleet vehicles in route choice urban games. Part I. Inverse fleet assignment theory.}
\author[1]{Grzegorz Jamr\'oz$^*$}
\author[1]{Rafa{\l}  Kucharski}
\affil[1]{Faculty of Mathematics and Computer Science, Jagiellonian University, Kraków, Poland}
\affil[*]{Corresponding author, e-mail: grzegorz.jamroz@uj.edu.pl}
\setcounter{Maxaffil}{0}
\renewcommand\Affilfont{\itshape\small}
\maketitle

\noindent {\bf Abstract:} Detection of collectively routing fleets of vehicles in future urban systems may become important for the management of traffic, as such routing may destabilize urban networks leading to deterioration of driving conditions. Accordingly, in this paper we discuss the question whether it is possible to determine the flow of fleet vehicles on all routes given the fleet size and behaviour as well as the combined total flow of fleet and non-fleet vehicles on every route. We prove that the answer to this Inverse Fleet Assignment Problem is 'yes' for myopic fleet strategies which are more 'selfish' than 'altruistic', and 'no' otherwise, under mild assumptions on route/link performance functions. To reach these conclusions we introduce the forward fleet assignment operator and study its properties, proving that it is invertible for 'bad' objectives of fleet controllers.  We also discuss the challenges of implementing myopic fleet routing in the real world and compare it to Stackelberg and Nash routing. Finally, we show that optimal Stackelberg fleet routing could involve highly variable mixed strategies in some scenarios, which would likely cause chaos in the traffic network. 

\section{Introduction}

Year 2025: Technological advances of recent years have enabled the introduction of self-driving cars \cite{Maurer} as well as communication technologies such as vehicle-to-vehicle, V2V and vehicle-to-infrastructure, V2I, see \cite{V2X} for a survey. Many cities across the world \cite{YourCity, EU, China} have started adopting these technologies and it seems inevitable that the share of connected and autonomous vehicles (CAVs) will increase, making the daily commute easier and more comfortable for all the drivers. Year 2045: As the time has been passing by, some cities have started discovering that even though the technology has enabled autonomous driving indistinguishable from human driving and the road infrastructure and number of vehicles is unchanged, the driving conditions have deteriorated. What has gone wrong? Is this a possible future scenario? Can it be prevented?

In \cite{JamrozSciRep} it was shown that coordinated routing strategies applied by a fleet of CAVs, which otherwise drive exactly as human drivers, may, for certain plausible myopic fleet objectives, result in deterioration of driving conditions across the system or for the human drivers (HDVs), while the travel times of CAVs are improved. Looking for ways to address this undesirable outcome, in this paper we set out to assess the capabilities a city may require to detect fleets of collectively routing vehicles (CRVs), which may be formed of coordinated fleets of CAVs or subscribers of route guidance systems, including systems such as Google Maps,  whose inner workings are undisclosed, however many drivers unconditionally follow their instructions.
To this end, we introduce and solve the inverse fleet assignment problem, which consists in identifying the flows of fleet vehicles on routes/links based on measurement of total flows. 
We defer the discussion of detection of individual drivers to \cite{JamrozDetectionInPrep}. To fix the ideas, as mentioned above, in the mixed HDV-CRV traffic system we can interpret the CRVs as:
\begin{itemize}
\item \emph {Commercial guidance system}, where a given fraction of drivers is subscribed to a routing system and follows its guidance exactly. The guidance is based on a predefined strategy/behaviour.
\item \emph {Fleet of CAVs}, consisting of a given share of CAVs which select routes according to a pattern dictated by a centralized commercial fleet controller, which applies a given collective strategy, termed fleet behaviour. 
\end{itemize}
In both cases, the identities of drivers of CRVs are unknown to the city. The city is assumed, however, to know the number of participants and the collective strategy adopted by CRVs.

Note that in contrast to scenarios considered in the literature previously \cite{Kashmiri, VanVuren}, in the former case the guidance system is \emph{not} meant to be operated by the city aiming to optimize traffic, but by an independent commercial provider, who may be obligated to reveal some details about the functioning of the system. In the latter case, every traveller owns a CAV. 

The paper is organised as follows. In the remaining part of this section we introduce the setting and 
the inverse fleet assignment problem.  
In Section \ref{Sec_Inverse} we discuss the properties of fleet assignment involving non-standard behaviours and in Section \ref{Sec_Inverse2} we formulate, prove and discuss the main technical result of the paper -- the inverse fleet assignment theorem as well as related results. 
In Section \ref{Sec_Discussion} we summarize the obtained results and discuss the impact and future work. 
Appendix \ref{Sec_Stackelberg} discusses the relation of the myopic routing considered in this paper to Stackelberg and Nash routing and demonstrates that optimal Stackelberg routing may involve mixed strategies. 
Appendix \ref{Sec_convconc} contains proofs omitted from the main exposition and Appendix \ref{Sec_Examples} provides multiple examples of networks for which the inverse fleet assignment theory is or is not applicable.

\subsection{Notation and conventions}
\begin{itemize}
\item $R$ available routes are denoted $r_1, r_2, \dots, r_R$ and numbered $1,\dots, R$.
\item The $R$-dimensional flow vectors on routes are denoted $\bold{q}$ (or $\bold{q^{CRV}}$, $\bold{q^{HDV}}$, etc.) such that $\bold{q} = (q_1, q_2, \dots, q_R)^T$. The total flow is denoted by $|\bold{q}| := q_1 + q_2 + \dots q_R$. If $\bold{q}, \bold{q'}$ are two flow vectors, $\bold{q} \le \bold{q'}$ means that simultaneously $q_1 \le q'_1, \dots, q_R \le q_R'$.  
\item For two vectors $\bold{q}, \bold{q'}$ the dot product is defined by $\bold{q} \cdot \bold{q'}:= \bold{q}^T \bold{q'} = \sum_{r=1}^R q_r q_r'$. 
\item The route delay functions are denoted by $t_1, t_2, \dots$ (i.e. $t$ with an arabic numeral subscript) while link delay functions are denoted either $t_a, t_b, t_c, \dots$, i.e. $t$ with a Latin lowercase letter subscript or $\boldsymbol{\tau}_{\alpha} = (\tau_{\alpha})$, where $\alpha$ stands for links. 
\item The link (arc) flows are denoted $\bold{a}$ (or $\bold{a^{HDV}}$ etc.) and are $A$-dimensional vectors, where $A$ is the number of links (arcs) in the network. 
\end{itemize}

\subsection{Problem setting and Fleet Objective Function}
\label{Sec_FOF}
In the main line of exposition we 
assume that, for simplicity,  
\begin{itemize}
\item the traffic system has one Origin-Destination (OD) pair, see Section \ref{Sec_Gen} for extension to multiple OD settings,
\item there are multiple routes available to drivers, $r = 1,2, \dots, R$ connecting Origin with Destination,
\item the travel time via route $r$, $r = 1,\dots, R$, is given by the delay function $$t_r(q_1, q_2, \dots, q_R),$$ where $q_s$ for $s = 1,2, \dots, R$ is the total flow (number of vehicles) selecting route $s$ on a given day,
\item background traffic, if any, is incorporated into functions $t_r$ and is not part of the flows $q_r$.
\end{itemize}
\noindent In this setting we assume that, every day, $q^{HDV}$ human driven vehicles choose among routes $1, \dots, R$, resulting in HDV flow $\bold{q^{HDV}} = (q_1^{HDV}, \dots, q_R^{HDV})$  and upon this the CRV fleet controller assigns $q^{CRV}$ vehicles to routes, resulting in a split $\bold{q^{CRV}} = (q^{CRV}_1, \dots, q^{CRV}_R)$ such that $|\bold{q^{CRV}}| = q^{CRV}$ and the fleet's myopic (one-day) objective 
\begin{equation}
F(q^{CRV}_1, \dots, q^{CRV}_R) = \lambda^{HDV} T^{HDV} + \lambda^{CRV} T^{CRV}
\label{eq_obj} 
\end{equation}
is minimized. In \eqref{eq_obj} 
\begin{eqnarray*}
T^{HDV} &=& \sum_{r = 1}^R q^{HDV}_r t_{r} \left(\bold{q^{HDV}} + \bold{q^{CRV}}\right),\\
T^{CRV} &=& \sum_{r=1 }^R q^{CRV}_r t_{r} \left(\bold{q^{HDV}} + \bold{q^{CRV}}\right).
\end{eqnarray*}
We consider various behaviours of the fleet, Fig. \ref{Fig_lambdas}, dependent on the values of parameters $\lambda^{HDV}, \lambda^{CRV}$. For instance, see \cite{JamrozSciRep}, 
\begin{itemize}
\item for $(\lambda^{HDV}, \lambda^{CRV}) = (0,1)$ the fleet is selfish as it minimizes its own collective travel time, 
\item $(\lambda^{HDV}, \lambda^{CRV}) = (1,0)$ corresponds to an altruistic fleet minimizing HDVs' travel time
\item $(\lambda^{HDV}, \lambda^{CRV}) = (-1,0)$ accounts for a malicious fleet, which maximizes HDVs' travel time 
\item $(\lambda^{HDV}, \lambda^{CRV}) = (1,1)$ expresses the social goal of minimization of total travel time 
\item  $(\lambda^{HDV}, \lambda^{CRV}) = (-1,1)$ corresponds to a competitive/disruptive fleet. 
\end{itemize}
We dismiss \cite{JamrozSciRep} $\lambda^{CRV}<0$ as corresponding to a fleet with self-harming behaviour component, compare \cite{SchwartingPNAS}, which inspired Fig. \ref{Fig_lambdas}, nevertheless the results presented in this paper can be used for $\lambda^{CRV}<0$ as well. Here a couple of remarks are due:
\begin{enumerate}
\item[i)] Objective \eqref{eq_obj} assumes that the fleet controller knows the flow human drivers' choices are going to result in {\bf before} deciding the proportions of fleet vehicles which are to be routed via different alternatives. This assumption,  although hardly fully realistic, seems to be a well-justified baseline, see Appendix \ref{Sec_Stackelberg} for discussion and comparison with Stackelberg and Nash routing, and a reasonable approximation in systems where human reactions to changing driving conditions are slow and take many days.  
\item[ii)] Objective \eqref{eq_obj} is convex and admits a unique minimizer only for certain combinations of $(\lambda^{HDV}, \lambda^{CRV})$. Hence, in contrast to the standard objectives used in transportation, such as system optimum or user equilibrium, in our setting one can guarantee the existence but {\bf not uniqueness} of the fleet objective minimizer, see Section \ref{Sec_Inverse} and Figure \ref{Fig_convexity}. In particular the standard gradient-based algorithms such as Frank-Wolfe (see \cite{FrankWolfe, BoylesBook}) are insufficient for finding the global minimizer and other algorithms are necessary, see Section \ref{Sec_realworld}. 
\end{enumerate}

\begin{figure*}[h!]
\centering
\includegraphics[scale=0.7]{"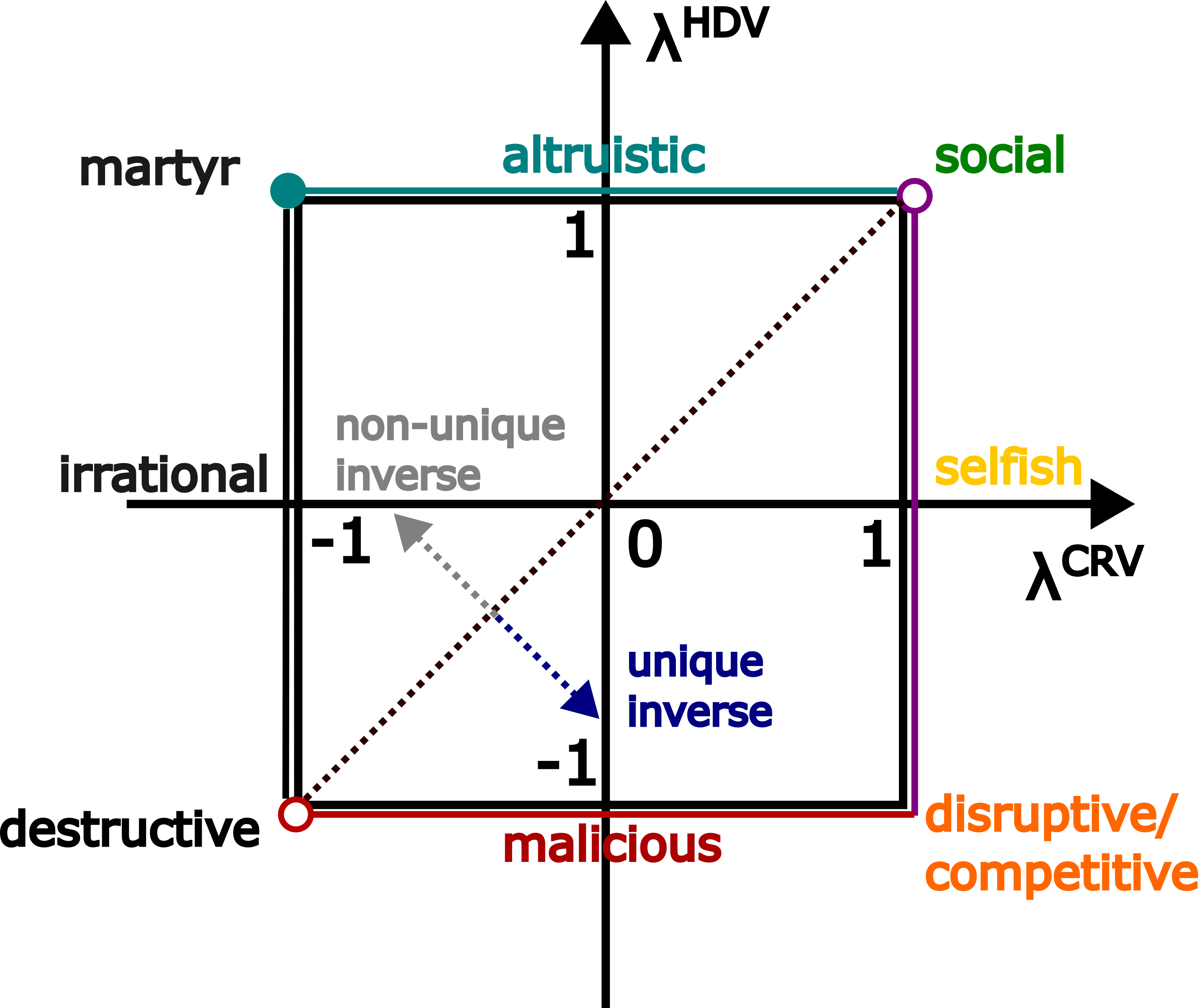"}
\caption{Different fleet objective functions can be presented on a square, with $|\lambda^{HDV}|, |\lambda^{CRV}| \le 1$ and $|\lambda^{HDV}|=1$ or $|\lambda^{CRV}|=1$ . The main technical result of the paper states that the fleet assignment operator can be inverted whenever $\lambda^{CRV} > \lambda^{HDV}$. Consequently, for all rational fleet behaviours ($\lambda^{CRV} \ge 0$) the fleet flows on routes can be recovered if the fleet strategy is more selfish than altruistic (e.g. malicious, disruptive, selfish). While this is impossible for good, pro-social behaviours such as social and altruistic, it seems to be less useful in the detection problem.} 
\label{Fig_lambdas}
\end{figure*}

\subsection{The inverse fleet assignment problem}
Here we state the inverse fleet assignment problem in the continuous setting, with flows allowed to be real-valued, which is a reasonable approximation in large systems. We discuss the discrete version in Section \ref{Sec_discrete}.

\begin{problem}[Inverse fleet assignment problem, continuous]
\label{problem_inverse}
Suppose the total flow in the system on a given day is given by $\bold{q} = (q_1, \dots, q_R)$ and the coefficients $\lambda^{HDV}, \lambda^{CRV}$ of the fleet strategy in \eqref{eq_obj} as well as fleet size $q^{CRV}$ are known. For which $\lambda^{CRV}, \lambda^{HDV}, t_r$ is it possible to determine the flow $\bold{q^{CRV}}$? Is the mapping $\bold{q} \mapsto \bold{q^{CRV}}$ (Lipschitz) continuous? Equivalently, is the inverse mapping $\bold{q} \mapsto \bold{q^{HDV}}$ (Lipschitz) continuous?
\end{problem}

Before embarking on the solution of Problem \ref{problem_inverse} we provide some examples which show that for certain fleet behaviours the inverse may not be 
unique.

\begin{example}[Social fleet]
\label{ex_social}
Let $\bold{q^{SO}}$ be the system optimum for a given total number of vehicles. Suppose that $0 < q^{CRV} < |\bold{q^{SO}}|$. Then for any $\bold{q^{HDV}}$ such that $\bold{0} \le \bold{q^{HDV}} \le \bold{q^{SO}}$ a social fleet assigns vehicles so as to achieve system optimum, i.e.  by $\bold{q^{CRV}} := \bold{q^{SO}} - \bold{q^{HDV}}$. Therefore, the mapping $\bold{q} \mapsto \bold{q^{CRV}}$ is not single-valued at $\bold{q} = \bold{q^{SO}}$, as viable distinct $\bold{q^{CRV}}$ can arise from multiple $\bold{q^{HDV}}$.
\end{example}
\begin{example}[Altruistic fleet]
\label{ex_altruistic}
Suppose the fleet's objective is to minimize human drivers' travel time. Let there be three independent routes and $\bold{q} = (30, 30, 40)$. If $q^{CRV} = 30$ then for $\bold{q^{HDV}} = (30, 0, 40)$ necessarily $\bold{q^{CRV}} = (0, 30, 0)$ as the fleet aims not to increase human drivers' travel time. Similarly, for $\bold{q^{HDV}} = (0, 30, 40)$ we obtain $\bold{q^{CRV}} = (30, 0, 0)$. Consequently, the mapping $\bold{q} \mapsto \bold{q^{CRV}}$ may be multi-valued for altruistic fleets.
\end{example}

Examples \ref{ex_social} and \ref{ex_altruistic} demonstrate that determining the assignment of fleet vehicles based on fleet size and objective fails for social and altruistic fleets. Theorem \ref{th_inverse} shows that for rational objectives given by \eqref{eq_obj} these are essentially the only cases when the mapping $\bold{q} \mapsto \bold{q^{CRV}}$ may not be single-valued, compare Fig. \ref{Fig_lambdas}.

\subsection{Background}
There is a vast literature on autonomous driving. Nevertheless, the coexistence of (fleets of) autonomous vehicles and human drivers has been studied typically from the point of view of microscopic driving, see e.g. \cite{Gora2020, Farah2022} or evolutionary game theory \cite{Bitar2022}. The game-theoretical route choice behaviour in mixed systems involving various players has been addressed e.g. in \cite{Yang2007, Harker1988}. The detection of cooperating fleets of vehicles has been studied for car following and platooning in \cite{Charlottin2024}.  However, it seems that the identification of group players has not been addressed in the route choice context. 

Inverse problems, see e.g. \cite{Groetsch}, have been known in science for a long time, and first rigorous solutions date back to the early 20th century when Hermann Weyl \cite{Weyl} deduced the asymptotic growth of eigenvalues of the Laplace operator from the area of the domain, which later gave rise to the celebrated 'Can one hear the shape of a drum?' problem \cite{Kac}. 

Solving inverse problems allows one to deduce the cause (e.g. initial condition, parameters) given the observation. For instance, observation of the gravitational or electric field may allow one to deduce the distribution of mass or electric charge and 'inverting' the diffraction image or signals emitted by protons are the basis of spectroscopic methods in physics and chemistry, e.g. \cite{Woolfson}  or tomography methods in medicine, e.g. \cite{Morris}. Many inverse problems, however, are ill-posed, e.g. the unstable inverse heat equation, in which the initial temperature distribution is to be deduced from observation of this temperature at a later point in time, or indeed the shape of a drum which is in general not uniquely determined by the set of the eigenvalues of the Laplace operator \cite{GordonDrum}. 

The most important inverse problem in transportation engineering is the OD demand matrix estimation from observed link flows, see e.g. \cite{Bierlaire, Cascetta, Krylatov, VanZulyen}. In this approach the observer is provided with link counts (congestion levels) from detectors and tries to estimate the OD demand matrix, typically under the assumption of user equilibrium. Our setting bears some resemblance to it in the way that we begin with link flows. However, we do not estimate the OD demand; in fact the crux of the argument is presented on a single Origin-Destination pair. What we instead strive to estimate are the flows of fleet vehicles under known total flows on routes/links, and \emph {known} OD demand of CRVs and HDVs. Another stream of research is related to inferring the network state from transportation data \cite{NetworkLearning}, such as the shortest path from some observations \cite{BurtonToint}. 
We are not aware of any literature on inverse problems of fleet assignment in a setting similar to ours, see e.g. \cite{Drayage} for a different context of drayage, since, the problem being futuristic enough, it has not been considered in classical transportation settings. Therefore, the results presented in this paper are, to the authors knowledge, the first of a kind and open a new area of research. 

Convex optimization methods in transportation are studied in detail e.g. in \cite{BoylesBook} where also standard algorithms, such as Frank-Wolfe constrained optimization \cite{FrankWolfe} for determining user equilibrium or system optimum, which result in convex functionals, are discussed. However, for more elaborate assignment objectives,  efficient algorithms seem to be unexplored. It is the more so for the inverse problems, see however \cite{InverseOpt} for a recent survey of inverse optimization methods and \cite{Heuberger} for the inverse combinatorial problems such as the inverse shortest path problem, where the goal is to find costs of arcs, closest possible to the a priori estimates, such that the observed paths selected by travelers are the shortest.

Finally, let us mention the more general Inverse Optimal Control \cite{Kalman} and Inverse Reinforcement Learning, \cite{IOCIRL, Deshpande, IRL}, which typically focus on recovering the reward function, rather than initial state/parameters, from observed sequential/continuous behaviour and are more suited to problems with abundant data, e.g. \cite{ScalableIRL}, in contrast to the static limited-data one-off inverse problem considered in this paper. 

\section{Fleet assignment theory}

\label{Sec_Inverse}

In this section we discuss the basic ideas related to myopic fleet assignment and introduce the fleet assignment operators. We study their properties, providing multiple examples to facilitate understanding. We remark that the notions related to networks, which we utilize, such as route flow, link flow, link-route incidence matrix, route-flow to link-flow conversion as well as system-optimal (yet not with general objectives) optimization in both route-flow and link-flow formulation are standard in transportation theory, see e.g. \cite{CascettaBook, BoylesBook}.
We also note that here we use the term 'operator' to represent a mapping between two spaces, not to be confused with the fleet operator/controller who owns/controls the fleet and prescribes its behaviour. In the end, however, the fleet operator/controller may indeed use the image of the fleet assignment operator to assign the flows to routes. For clarity, in the following we avoid using the term 'fleet operator' in favour of 'fleet controller' and '(fleet) assignment operator' is understood as a mapping between vector spaces. 

\subsection{\bf Fleet objective function}
\noindent Reformulating \eqref{eq_obj} we define the fleet objective function as 
\begin{equation}
\label{eq_objective}
F(\bold{h}, \bold{f}) = (\lambda^{HDV}\bold{h} + \lambda^{CRV}\bold{f})\cdot \bold{t}(\bold{h} + \bold{f})
\end{equation}
which, for a given HDV flow $\bold{h} \in [0,\infty)^R$, assigns to every fleet assignment $\bold{f} \in [0,\infty)^R$ a generalized cost $F(\bold{h}, \bold{f})$. In \eqref{eq_objective}
$$\bold{t}(\bold{q}) := (t_1(\bold{q}), \dots, t_R(\bold{q}))^T$$
is the R-vector of travel times
and $\bold{f}\cdot \bold{t}$ and $\bold{h}\cdot \bold{t}$ are dot products.

\begin{remark}
\label{Rem_continuity}
Continuity properties of $F$ are the same as continuity properties of $\bold{t}$. For instance, if $\bold{t}$ is Lipschitz continuous then so is $F(h,\cdot)$. 
\end{remark}

\subsection{\bf Fleet Assignment Operators}
\begin{definition}[Route Fleet Assignment Operator]
For given $\lambda^{HDV}, \lambda^{CRV}$ and fleet size $q^{CRV}$ we define the \emph{Route Fleet Assignment Operator} $G: [0,\infty)^R \to 2^{[0,\infty)^R}$ by
\begin{equation*}
G\left(\bold{q^{HDV}}\right) = \arg \min_{\bold{f}} \left\{F(\bold{q^{HDV}},\bold{f}): |\bold{f}| = q^{CRV}\right\}.
\end{equation*} 
The Route Fleet Assignment Operator $G$ assigns to every $\bold{q^{HDV}}$ the set of minimizers of the fleet objective function. Using $G$, we define the operator $H$, which maps the HDV flow to the total flow by 
\begin{equation*}
H := Id + G,
\end{equation*}
i.e.
\begin{equation*}
H\left(\bold{q^{HDV}}\right) = \bold{q^{HDV}} + G\left(\bold{q^{HDV}}\right).
\end{equation*}
\end{definition}

\begin{definition}
Let $I \in \{0,1\}^{R \times A}$ be the \emph {link-route incidence matrix} defined by $I^{r\alpha} = 1$ if link $\alpha$ belongs to route $r$ and $I^{r\alpha} = 0$ otherwise.
\end{definition}

\begin{definition}
Let $\Lambda : [0,\infty)^R \to [0,\infty)^A$ be the route-flow to link-flow conversion operator defined by
\begin{equation*}
\Lambda(\bold{q})_\alpha = \sum_{r=1}^R I^{r\alpha} q_r.
\end{equation*}
\end{definition}

\begin{definition}[Link Fleet Assignment Operator]
\label{def_linkassignoper}
Suppose the delay functions $t_r$ on every available route $r$ are given as sums of costs on links, i.e. $t_r(\bold{q}) = \sum_{\alpha} I^{r\alpha} \tau_\alpha(\Lambda(\bold{q}))$ where $\tau_\alpha$ are differentiable link-delay functions (e.g. BPR functions). Consider the link objective function
\begin{equation*}
\Phi(\boldsymbol{\eta}, \boldsymbol{\phi}) = (\lambda^{HDV}\boldsymbol{\eta} + \lambda^{CRV}\boldsymbol{\phi})\cdot \boldsymbol{\tau}(\boldsymbol{\eta} + \boldsymbol{\phi}).
\end{equation*}
For given $\lambda^{HDV}, \lambda^{CRV}$ and fleet size $q^{CRV}$ we define the \emph{Link Fleet Assignment Operator} $\Gamma: [0,\infty)^A \to 2^{[0,\infty)^A}$ by
\begin{equation*}
\Gamma \left(\boldsymbol{\eta^{HDV}}\right) = \arg \min_{\boldsymbol{\phi}} \left\{\Phi(\boldsymbol{\eta^{HDV}},\boldsymbol{\phi}): \exists \bold{f}: |\bold{f}| = q^{CRV} \mbox{ and } \boldsymbol{\phi} = \Lambda(\bold{f})\right\}.
\end{equation*} 
The Link Fleet Assignment Operator identifies the set of optimal fleet link flows given the HDV link flow. This link flow may potentially originate from various fleet route flows. 
\end{definition}

\begin{proposition}[Convexity of domain]
\label{Prop_convexity}
\,

\begin{enumerate}
\item[i)] The set $\{\bold{q} \in [0,\infty)^R:  |\bold{q}| = q^{CRV}\}$ is convex.
\item[ii)] The set $\{\bold{a} \in [0,\infty)^A: \exists \bold{f}: |\bold{f}| = q^{CRV} \mbox{ and } \bold{a} = \Lambda(\bold{f})\}$ is convex.
\end{enumerate}
\end{proposition}
\begin{proof}
If $\bold{a} = \Lambda(\bold{f})$, $\bold{\tilde{a}} = \Lambda(\tilde{\bold{f}})$ and $|\bold{f}| = |\bold{\tilde{f}}| = q^{CRV}$ then for every $c \in [0,1]$ we have $$c\bold{a}+(1-c)\bold{\tilde{a}} = c\Lambda(\bold{f}) + (1-c)\Lambda(\bold{\tilde{f}}) = \Lambda(c \bold{f} + (1-c) \bold{\tilde{f}}),$$ where we used the linearity of operator $\Lambda$. We conclude by noting that $$|c \bold{f} + (1-c) \bold{\tilde{f}}| = c|\bold{f}| + (1-c)|\bold{\tilde{f}}| = q^{CRV},$$ which proves both i) and ii).  
\end{proof}

\begin{proposition}[Properties of Fleet Assignment Operators]
Assume that the route/link delay functions $t_r, \tau_\alpha$ are continuous, increasing and convex. Then 
\begin{enumerate}
\item [i)] $G(\bold{q^{HDV}})$ is non-empty for every $\bold{q^{HDV}}$. 
\item [ii)] $\Gamma(\boldsymbol{\eta^{HDV}})$ is non-empty for every $\boldsymbol{\eta^{HDV}}$. 
\item [iii)] If $G(\bold{q^{HDV}})$ contains one element then the corresponding $\Gamma(\boldsymbol{\eta^{HDV}})$ is a one-element set. 
\end{enumerate}
\end{proposition}

\begin{proof}
i) and ii) are an immediate consequence of compactness of the set $\{\bold{f} : |\bold{f}| = q^{CRV} \}$. iii) follows by the fact that the operator $\Lambda$ of route-flow to link-flow conversion is single-valued.
\end{proof}

\begin{proposition}[Fleet optimization for independent links]
\label{Prop_foil}
Suppose that the link travel times are independent, i.e. 
\begin{equation*}
\Phi(\boldsymbol{\eta}, \boldsymbol{\phi}) = \sum_{\alpha} (\lambda^{HDV} \eta_{\alpha} + \lambda^{CRV} \phi_{\alpha}) \tau_\alpha (\eta_\alpha + \phi_\alpha) = \sum_\alpha \Phi_\alpha({\eta}_\alpha, {\phi}_\alpha)
\end{equation*}
and $\tau_{\alpha}$ are continuously differentiable, strictly increasing and convex. Then, see Fig. \ref{Fig_convexity} (see also Figs. \ref{Fig_convexity_example}, \ref{Fig_convexity_example3routes} for examples),
\begin{enumerate}
\item [i)] $\Phi(\boldsymbol{\eta}, \cdot)$ is convex if $\lambda^{HDV}, \lambda^{CRV}\ge 0$ and $\lambda^{HDV} + \lambda^{CRV}>0$. Consequently, $\Gamma(\boldsymbol{\eta^{HDV})}$ is a one-element set, i.e. to every $\boldsymbol{\eta^{HDV}} = \Lambda(\bold{q^{HDV}})$ it assigns a unique fleet link flow.
\item [ii)] $\Phi(\boldsymbol{\eta}, \cdot)$ is concave if  $\lambda^{HDV}, \lambda^{CRV}\le 0$ and $\lambda^{HDV} + \lambda^{CRV}<0$. Consequently, the minimizers of $\Phi$ may be nonunique.
\item [iii)] If $\lambda^{HDV}<0$ and $\lambda^{CRV} > 0$ then $\Phi(\boldsymbol{\eta}, \cdot)$ may be convex, concave or either locally convex or concave depending on the relative values of $\lambda^{HDV}, \lambda^{CRV}$ and values of flows $\boldsymbol{\eta}, \boldsymbol{\phi}$. 
\end{enumerate}
\end{proposition}
\begin{figure*}[h]
\centering
\includegraphics[scale=0.7]{"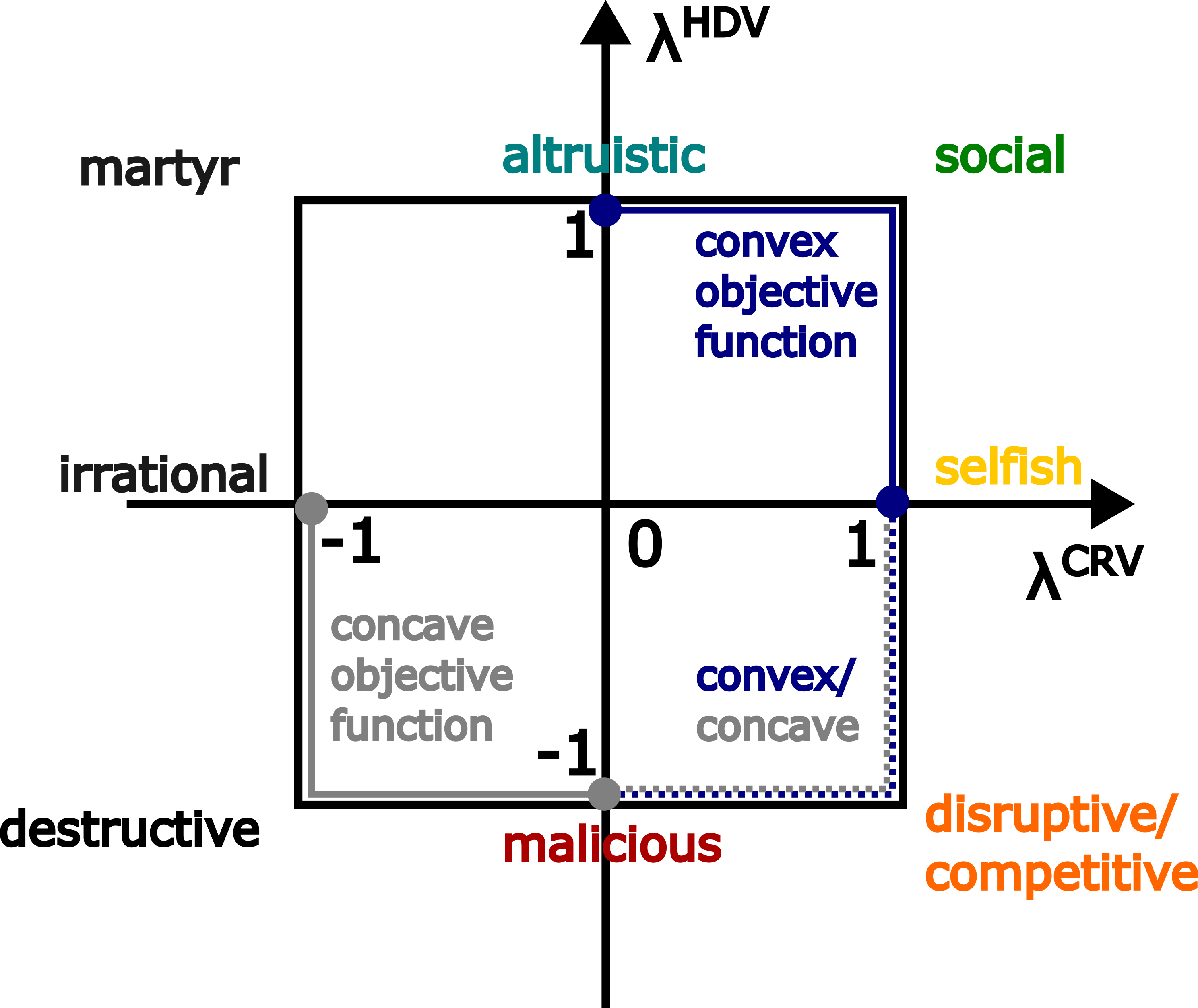"}
\caption{Different combinations of $(\lambda^{HDV}, \lambda^{CRV})$ result in the link objective function being convex or concave or locally convex/concave dependent on the relative values of fleet and HDV flows. Therefore, a global minimizer of fleet objective, unique in convex cases, is no longer guaranteed to be unique across a large part of the behavioral spectrum, e.g. for competitive and malicious objectives. The link delay functions are assumed to be increasing and convex.} 
\label{Fig_convexity}
\end{figure*}

\begin{figure*}[h]
\hspace{-0.9cm}
\includegraphics[scale=0.6]{"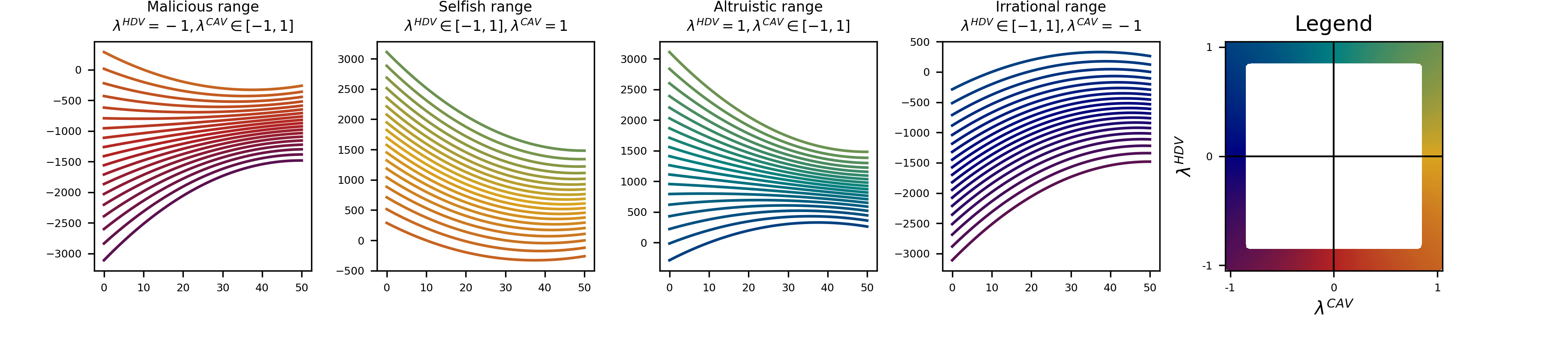"}
\caption{Objective function \eqref{eq_objective} shape for different fleet strategies corresponding to different couples $(\lambda^{HDV}, \lambda^{CRV})$. Here, we assumed two independent routes with delay functions given by $\bold{t}(\bold{q}) =  (t_1(q_1), t_2(q_2)) = (5(1 + (q_1/50)^2), 15(1 + (q_2/80)^2))$, fixed HDV flow $\bold{h} = (10, 40)$ and $50$ fleet vehicles.  Vertical axis -- objective function value. Horizontal axis -- flow $f_1$ (e.g. value $5$ corresponds to fleet flow $\bold{f} = (5,45)$). The objective is concave for the malicious strategy, changes into convex for certain strategy in the disruptive area, is convex for selfish, social and altruistic strategies, compare Fig. \ref{Fig_convexity}} 
\label{Fig_convexity_example}
\end{figure*}

\begin{figure*}[h!]
\hspace{-0.7cm}
\includegraphics[scale=0.55]{"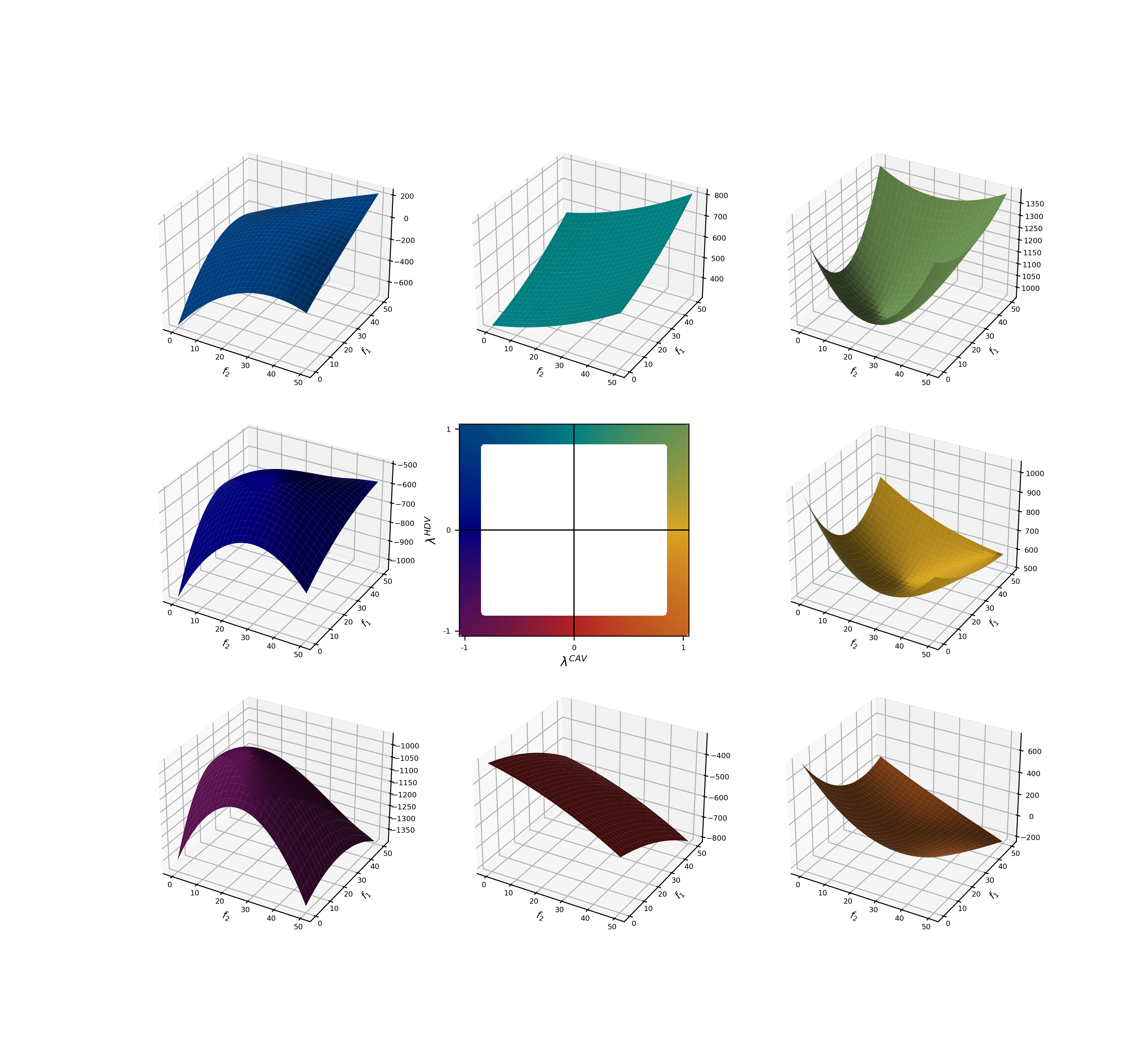"}
\caption{Objective function \eqref{eq_objective} shape for different fleet strategies corresponding to different couples $(\lambda^{HDV}, \lambda^{CRV})$. Here, we assumed three independent routes with delay functions given by $\bold{t}(\bold{q}) =  (t_1(q_1), t_2(q_2), t_3(q_3)) = (5(1 + (q_1/50)^2), 5(1 + (q_2/50)^2), 15(1 + (q_3/80)^2))$, fixed HDV flow $\bold{h} = (25, 25, 0)$ and $50$ fleet vehicles.  Vertical axis -- objective function value. Horizontal axes -- flows $f_1$ and $f_2$ (e.g. $f_1 = 5$, $f_2 = 10$ corresponds to fleet flow $\bold{f} = (5, 10, 35)$). The feasible set is given by $\{f_1 + f_2 \le 50\}$ in order to guarantee non-negativity of flow $f_3$. The objective is concave for the malicious strategy, and convex for selfish, social and altruistic strategies, compare Fig. \ref{Fig_convexity}. The profiles for strategies for $\lambda^{CRV} < 0$ can be obtained from the profiles for $\lambda^{CRV} > 0$ by taking the negative of the respective objective function.}
\label{Fig_convexity_example3routes}
\end{figure*}

\begin{proof}
\begin{enumerate}
\item [i)] If both $\lambda^{HDV}$ and $\lambda^{CRV}$ are non-negative and at least one of them is positive then $\Phi_\alpha$ is convex for every $\alpha$. Consequently, $\Phi$ is convex and $\Gamma$ minimizes $\Phi$ on a convex set. Therefore, the minimizer is unique.   
\item [ii)] Concavity follows by i) applied to $-\Phi$. The minimizers (perhaps non-unique) of a concave function are assumed at the boundary of the feasible set. 
\item [iii)] See Appendix \ref{Sec_convconc}.
\end{enumerate}
\end{proof}

\begin{proposition}
If the routes are linearly dependent, i.e. the link-route incidence matrix is not of maximal rank, then the Route Fleet Assignment Operator $G$ is in general multi-valued. 
\end{proposition}
\begin{proof}
Suppose there exists a real vector $\bold{a}$ such that $\sum a_r I^{rl} = 0$ for every $l$. Suppose that $\bold{q^{CRV}} \in G(\bold{q^{HDV}})$ has positive entries. Then there exists $\varepsilon > 0$ such that $\bold{q^{CRV}} + \varepsilon \bold{a}$ has positive entries and it yields the same link flow as $\bold{q^{CRV}}$.  
\end{proof}

Next we study the continuity of the forward fleet assignment operators while the continuity of the inverse is established in Theorem \ref{Th_continuity}. We restrict ourselves to the set $|\bold{q^{HDV}}| = const$. We begin by showing an example of lack of continuity resulting in 'switching' the local minimum.

\begin{example}
\label{Ex_switching}
Let there be two equivalent independent routes. For the malicious strategy and $\bold{q^{HDV}} = (25,25)$ we have $G(\bold{q^{HDV}}) = \{(50,0), (0,50)\}$, i.e. there are two minimizers. However, for $\bold{q^{HDV, \varepsilon}} = (25+\varepsilon,25-\varepsilon)$, where $\varepsilon > 0$, the unique minimizer is  $G(\bold{q^{HDV, \varepsilon}}) = \{(50,0)\}$. On the contrary, for $\varepsilon < 0$ we have $G(\bold{q^{HDV, \varepsilon}}) = \{(0,50)\}$. Consequently, $G$ is discontinuous at $(25,25)$. 
\end{example}

\begin{definition}
We say that $\bold{f}$ is a strict local minimizer of $F$ if $$F(\bold{q^{HDV}}, \bold{\tilde{f}}) > F(\bold{q^{HDV}},\bold{f})$$ for $\|\bold{\tilde{f}} - \bold{f}\|$ small enough, where $\|\cdot\|$ is any  finite-dimentional norm (e.g. maximum).
\end{definition}

\begin{proposition}[Continuity of fleet assignment operator]
\label{Prop_continuityFA}
Let $\bold{f}$ be a strict local minimizer of $F$ and let $\bold{q^{HDV,n}}$ be a sequence of HDV flows satisfying $|\bold{q^{HDV,n}}| = |\bold{q^{HDV}}|$ and $\bold{q^{HDV,n}} \to \bold{q^{HDV}}$. Then 
$$G^{loc}\left(\bold{q^{HDV,n}}\right) \to G^{loc}\left(\bold{q^{HDV}}\right) = \{\bold{f}\}$$
as $n \to \infty$, where $G^{loc}$ is the restriction of operator $G$ to a small neighbourhood of $\bold{f}$.  

\end{proposition}

\begin{proof}
Let $K \ni \bold{f}$ be a small closed neighbourhood such that $F(\bold{q^{HDV}}, \bold{\tilde{f}}) > F(\bold{q^{HDV}},\bold{f})$ for $\bold{\tilde{f}} \in K.$ 
Suppose $\bold{f_n} \in G^{loc, K}(\bold{q^{HDV,n}})$. As $K$ is compact there exists a subsequence $\bold{f_{n_k}} \xrightarrow{k \to \infty} \bold{{f_0}} \in K$. Suppose $\bold{f_0} \neq \bold{f}$. Then, $\delta = F(\bold{q^{HDV}}, \bold{f_0}) - F(\bold{q^{HDV}}, \bold{f})$ is strictly positive by the fact that $\bold{f}$ is a strict local minimizer and by continuity of $F$ we have, for $k$ large enough, 
\begin{equation*}
F(\bold{q^{HDV, n_k}}, \bold{f_{n_k}}) > F(\bold{q^{HDV}}, \bold{f_0}) - \frac \delta 2 = F(\bold{q^{HDV}}, \bold{f}) + \frac \delta 2 
\end{equation*}
on the one hand and
\begin{equation*}
F(\bold{q^{HDV, n_k}}, \bold{f}) < F(\bold{q^{HDV}}, \bold{f}) + \frac \delta 2
\end{equation*}
on the other. Hence $$F(\bold{q^{HDV, n_k}}, \bold{f}) < F(\bold{q^{HDV, n_k}}, \bold{f_{n_k}}),$$ which contradicts the fact that $\bold{f_{n_k}}$ is a minimizer. 
\end{proof}

\begin{corollary}
Strict local minimizers are stable. This fails to hold if we drop the assumption of strictness. Globally, if there is more than one strict minimizer for which the objective values coincide then 'switching' of minimizer can occur, see Example \ref{Ex_switching}. Nevertheless, if there is only one strict local minimizer which is the case for strictly convex objectives, the fleet assignment operator $G$ is globally continuous as the set $K$ in the proof of Proposition \ref{Prop_continuityFA} can be taken as the whole feasible set. 
\end{corollary}

\subsection{'Real-world' assignment by fleet controller}
\label{Sec_realworld}
Real-world assignment of fleet vehicles is a discrete rather than continuous optimization problem. Below, we discuss the relation of the two approaches. A {\bf 'continuous' real-world fleet controller} would likely:
\begin{itemize}
\item Assign the vehicles according to the unique minimizer in the case of {\bf strictly convex} objective function (e.g. selfish, social), which can be identified by Frank-Wolfe or similar algorithms. 
\item Assign the vehicles using one of the finite number of local minimizers in the case of {\bf strictly concave} objective function (e.g. malicious). The minimizers can be found by inspecting the corners of the feasible set and can exhibit the same (by chance) or different objective values. In the latter case one may assume (although it is not necessary for further exposition) that the best local minimizer, i.e. the global minimizer is selected. 
\item Assign the vehicles using: brute force (in the case of small systems), metaheuristics \cite{Metaheuristics} or other algorithms in the case of {\bf non-convex and non-concave objectives}. The identified assignment will typically be a local minimum (perhaps non-strict) of the objective function, which is what we assume in this paper.  
\end{itemize}

\noindent A real-life {\bf 'discrete' fleet controller} could:
\begin{itemize}
\item Take (one of) the identified 'continuous' local minimizers and round up the flows. This approach is reasonable when the objective function is Lipschitz continuous with a not-too-large Lipschitz constant. Then the assignment will result in near optimal objective (whatever the optimal discrete objective may be). 

\item Find the actual discrete minimum over the discrete space. This approach results potentially in a better objective function value but is computationally very expensive and possibly prohibitive.
\end{itemize}

In this paper we assume that the fleet operator/controller uses the first approach as a computationally reasonable alternative, leaving the other option to further research.  The first approach allows us to assume that the actual discrete flow used is close to a continuous minimizer, see Section \ref{Sec_discrete}.  On the other hand, for the second approach (fleet controller using global discrete minimizers) we note that 
\begin{itemize}
\item The discrete global minimizer(s) does not necessarily lie close to the global continuous minimizer even in the case of strictly convex objective function, unless the objective is uniformly convex. 
\item The discrete global minimizers(s) does not have to lie close to any of the continuous local minimizers, even in the case of strictly concave objectives.
\item Even if the continuous minimizer is unique, the discrete minimizer does not have to be unique. 
\end{itemize}

Finally, we note that even if the fleet controller/operator assigns vehicles using a local (and not global) minimizer, the inverse theory still can be applied, see Remark \ref{rem_localmin}. 

\section{Inverse fleet assignment theorem and related results}
\label{Sec_Inverse2}
In this section we formulate and prove the main technical results of this paper, i.e. the theorem on inverse fleet assignment along with other auxiliary results.
\subsection{Inverse fleet assignment theorem}
\begin{definition}
\label{def_posdef}
A square real matrix $M \in \mathbb{R}^{N \times N}$ is \emph{positive definite} if for every real vector $\bold{v} \in \mathbb{R}^N$, $\bold{v} \neq 0$ we have $\bold{v}^T M \bold{v} >0$. 
\end{definition}
Note that in Definition \ref{def_posdef} we do not necessarily assume that the matrix $M$ is symmetric. As, however, $\bold{v}^T M \bold{v} = \bold{v}^T M^T \bold{v}$ for any square matrix $M$, we find that the value of the quadratic form depends only on the symmetric part of $M$, i.e. $\bold{v}^T M \bold{v} = \bold{v}^T ((M + M^T)/2) \bold{v}$. 

\begin{theorem}[Inverse of fleet assignment]
\label{th_inverse}
Let $\bold{t} = (t_1(\bold{q}), t_2(\bold{q})), \dots, t_R(\bold{q}))^T$ be continuously differentiable delay functions (travel times) of routes $1,\dots, R$, dependent on the flow of all vehicles $\bold{q} = (q_1, \dots, q_R)$ and let $q^{CRV}$ be the total flow of fleet vehicles (i.e. fleet size).  
If $\lambda^{HDV} <  \lambda^{CRV}$ then for every $\bold{q}$ such that the gradient matrix $\nabla \bold{t}(\bold{q})$ is positive definite there exists at most one $\bold{q^{CRV}}$ such that $\bold{q^{CRV}} \in G(\bold{q - q^{CRV}})$. In other words, for $\lambda^{HDV} <  \lambda^{CRV}$ and positive definite $\bold{\nabla t(q)}$ there exists at most one couple $(\bold{q^{HDV}}, \bold{q^{CRV}})$ such that $\bold{q^{HDV}} + \bold{q^{CRV}} = \bold{q}$ and $\bold{q^{CRV}}$ optimizes the fleet objective, which is furthermore equivalent to $H^{-1}(\bold{q})$ being an at most one-element set. 
\end{theorem}

\begin{proof}
The optimization by fleet can be expressed as
\begin{equation}
\label{eq_qcav}
\bold{q^{CRV}} \in \arg \min_{\bold{f}} \left\{F(\bold{q^{HDV}},\bold{f}): |\bold{f}| = q^{CRV}\right\}.
\end{equation} 
Now, given a total vehicle assignment $\bold{q}$, fleet size $q^{CRV}$ and fleet strategy we will show that inclusion
\begin{equation}
\label{eq_qCAV}
\bold{q^{CRV}} \in \arg \min_{\bold{f}} \left\{F(\bold{q} - \bold{q^{CRV}}, \bold{f}): |\bold{f}| = q^{CRV} \right\}
\end{equation}
has at most one solution $\bold{q^{CRV}}$ , i.e. $\bold{q}$ can arise from at most one couple $(\bold{q^{HDV}}, \bold{q^{CRV}})$ conditional on fleet strategy and size. 

Suppose first that $\arg \min_{\bold{f}} \left\{F(\bold{q} - \bold{f^*}, \bold{f}): |\bold{f}| = q^{CRV} \right\} = \bold{f^*}$. Then for every feasible direction  $\bold{g}$, i.e. $\bold{g}$ such that $\bold{f}^* + \delta \bold{g}$ has non-negative entries for $\delta$ small enough (viable flow) and $|\bold{g}| = 0$ (so that $|\bold{f}^* + \delta \bold{g}| = q^{CRV}$) we have
\begin{equation*}
\frac {d^+} {d\delta} F(\bold{q} - \bold{f^*}, \bold{f^*} + \delta \bold{g})|_{\delta = 0} \ge 0.
\end{equation*}
Computing the derivative we obtain 
\begin{equation*}
\lambda^{CRV} \bold{g}\cdot \bold{t(q)} + \left(\lambda^{HDV}(\bold{\bold{q} - \bold{f^*}}) + \lambda^{CRV}\bold{f^*}\right) \cdot (\nabla \bold{t}(\bold{q}) {\bold{g}})  \ge 0,
\end{equation*}
where $\bold{\nabla t}$ is the gradient matrix and $\nabla \bold{t}(\bold{q}) {\bold{g}} = \sum_{r=1}^R  (\partial_r \bold{t}(\bold{q})) g^r$ is the derivative of $\bold{t}$ in direction $\bold{g}$ evaluated at $\bold{q}$. Rearranging we obtain
\begin{equation}
\label{eq_der1}
\lambda^{CRV} \bold{g}\cdot \bold{t(q)} + \left((\lambda^{CRV}- \lambda^{HDV}) \bold{f^*} + \lambda^{HDV}\bold{q}\right) \cdot (\nabla \bold{t}(\bold{q}) {\bold{g}})  \ge 0. 
\end{equation}
Suppose now that $\bold{\tilde{f}} = \bold{f^*} + \bold{g}$ is another solution of \eqref{eq_qCAV}. Then the directional derivative of $F$ at $\bold{\tilde{f}}$ in a feasible direction $\bold{\tilde{g}}$ is non-negative and hence 
\begin{equation}
\label{eq_der2}
\lambda^{CRV} \bold{\tilde{g}}\cdot \bold{t(q)} + \left((\lambda^{CRV}- \lambda^{HDV}) \bold{\tilde{f}} + \lambda^{HDV}\bold{q}\right) \cdot (\nabla \bold{t}(\bold{q}) {\bold{\tilde{g}}})  \ge 0. 
\end{equation}
Setting $\bold{\tilde{g}} = -\bold{g}$ and adding \eqref{eq_der1} and \eqref{eq_der2} we obtain 
\begin{equation*}
(\lambda^{CRV} - \lambda^{HDV}) (\bold{f^*} - \bold{\tilde{f}})\cdot (\nabla \bold{t}(\bold{q}) {\bold{g}}) \ge 0.
\end{equation*}
Noting that $\bold{g} = \bold{\tilde{f}} - \bold{f^*}$ is a feasible direction in \eqref{eq_der1} and  $\bold{\tilde{g}} = -\bold{g}$ is feasible in \eqref{eq_der2} we obtain
\begin{equation}
\label{eq_last}
-(\lambda^{CRV} - \lambda^{HDV}) \bold{g}\cdot (\nabla \bold{t}(\bold{q}) {\bold{g}}) \ge 0.
\end{equation}
Consequently, if $\bold{\nabla t}(\bold{q})$ is positive definite and $\lambda^{CRV} > \lambda^{HDV}$, the solution of \eqref{eq_qCAV}, if it exists, is unique. 

\end{proof}

\begin{corollary}
Theorem \ref{th_inverse} implies that for every fleet strategy, which is more (own group) 'selfish' than 'altruistic' (i.e. e.g. for selfish, malicious, disruptive) the fleet assignment can be recovered from the assignment of all vehicles and the size of the fleet. On the other hand, e.g. for social ($\lambda^{HDV} = \lambda^{CRV} = 1$) or altruistic ($\lambda^{HDV} = 1$, $\lambda^{CRV} = 0$) this is not possible. 
\end{corollary}

\begin{remark}
\label{rem_feasibledir}
In Theorem \ref{th_inverse} it suffices to assume that $\bold{\nabla t(q)}$ is positive definite only with respect to feasible directions, i.e. $\bold{g}\cdot (\nabla \bold{t}(\bold{q}) {\bold{g}})>0$ for every $\bold{g} \neq \bold{0}$ such that $|\bold{g}|=\bold{0}$. 
\end{remark}

\begin{remark}
\label{rem_localmin}
Theorem \ref{th_inverse} assumes that the fleet controller is able to assign the flows on routes such that the result minimizes the objective function, i.e. for given (expected by the fleet controller) HDV assignment $\bold{q^{HDV}}$ it can identify at least one $\bold{q^{CRV}}$ such that \eqref{eq_qcav} is satisfied. The proof of the theorem, however, requires only that $\bold{f}^*$ and $\bold{\tilde{f}}$ be {\bf local} minimizers. Consequently, the theorem still holds true for fleet controllers which use algorithms that generate assignments minimizing the objective locally. Such algorithms (e.g. gradient methods) are easily available. In contrast, global optimization techniques often use metaheuristics \cite{Metaheuristics} or approximation algorithms \cite{ApproxAlgos} and typically their convergence to the global minimum cannot be guaranteed yet they often find a local minimum. For instance, for the malicious strategy (with concave objective function), among assignments which put all fleet vehicles on a route with non-zero HDV flow there may be multiple local minimizers. If the fleet uses an algorithm that selects a route to which it assigns all vehicles resulting in a local (yet not necessarily global) minimization, this assignment can still be detected. 
\end{remark}

\begin{theorem}[Continuity of the inverse fleet assignment operator]
\label{Th_continuity}
Let $\bold{t}$ be twice continuously differentiable, $\nabla \bold{t}$ be positive definite and  $\lambda^{HDV} < \lambda^{CRV}$. Fix the total flow $q$. Then the inverse fleet assignment operators $I - H^{-1} : \bold{q}\mapsto \bold{q^{CRV}}$ as well as $H^{-1}: \bold{q}\mapsto \bold{q^{HDV}}$ are continuous.  
\end{theorem}
\begin{proof}
Suppose $\bold{f}^* \in (I - H^{-1}) \bold{q}^*$ and $\bold{f}^\# \in (I - H^{-1}) \bold{q}^\#$ sastisfy $|\bold{f}^*| = |\bold{f}^\#| = q^{CRV}$ for $|\bold{q}^*| = |\bold{q}^\#| = q$. Then, denoting $L:= \lambda^{CRV} - \lambda^{HDV}$ we obtain, by 
\eqref{eq_der1}, 
\begin{equation*}
\lambda^{CRV} \bold{g^*}\cdot \bold{t(q^*)} + \left(L \bold{f^*} + \lambda^{HDV}\bold{q^*}\right) \cdot (\nabla \bold{t}(\bold{q^*}) {\bold{g^*}})  \ge 0 
\end{equation*}
for every feasible direction $\bold{g}^*$. Rearranging this inequality and writing a similar one for $\bold{q}^\#$ and feasible direction $\bold{g}^\#$ we obtain
\begin{equation*}
\begin{cases}
-L\bold{f}^* \cdot (\nabla \bold{t}(\bold{q}^*)  \bold{g^*}) &\le \lambda^{CRV} \bold{g^*} \cdot \bold{t}(\bold{q}^*) + \lambda^{HDV} \bold{q}^* \cdot \nabla \bold{t}(\bold{q}^*) \bold{g^*}\\
-L\bold{f}^\# \cdot (\nabla \bold{t}(\bold{q}^\#)  \bold{g^\#}) &\le \lambda^{CRV} \bold{g^\#} \cdot \bold{t}(\bold{q}^\#) + \lambda^{HDV} \bold{q}^\# \cdot \nabla \bold{t}(\bold{q}^\#) \bold{g^\#}
\end{cases}
\end{equation*}
Setting $\bold{g}:= -\bold{g}^\# = \bold{g}^* = \bold{f}^\# - \bold{f}^*$ and adding the above inequalities we obtain
\begin{eqnarray*}
L\bold{g} \cdot \nabla \bold{t}(\bold{q}^*) \bold{g} + L\bold{f}^\# \cdot (\nabla \bold{t}(\bold{q}^\#) - \nabla \bold{t}(\bold{q}^*)) \bold{g} &\le&  \lambda^{CRV} \bold{g}(\bold{t}(\bold{q}^*) - \bold{t}(\bold{q}^\#))\\ +  \lambda^{HDV} \bold{q}^*(\nabla \bold{t}(\bold{q}^*) - \nabla \bold{t}(\bold{q}^\#))\bold{g} &+& \lambda^{HDV} (\bold{q}^* - \bold{q}^\#) \nabla \bold{t}(\bold{q}^\#) \bold {g}
\end{eqnarray*}
with the five terms denoted by $A, B, C, D, E$ and the inequality assuming the form $A + B \le C + D + E$. Assuming that the delay function is twice continuously differentiable, we may thus estimate
\begin{eqnarray*}
L\rho(\nabla \bold{t}) \|\bold{g}\|_2 ^2 &\le& L\bold{g}^* \cdot \nabla \bold{t}(\bold{q}^*) \bold{g^*} = A,\\
|B| = |L\bold{f}^\# \cdot (\nabla \bold{t}(\bold{q}^\#) - \nabla \bold{t}(\bold{q}^*)) \bold{g}| &\le& L \|\bold{f}^\#\|_2 \|\nabla^2 \bold{t}\|_{2,\infty} \|\bold{q}^* - \bold{q}^\#\|_2 \|\bold{g}\|_2, \\
|C| = |\lambda^{CRV} \bold{g}(\bold{t}(\bold{q}^*) - \bold{t}(\bold{q}^\#))| &\le& |\lambda^{CRV}| \|\bold{g}\|_2 \|\nabla \bold{t}\|_{2,\infty} \|\bold{q}^* - \bold{q}^\#\|_2, \\
|D| = |\lambda^{HDV} \bold{q}^*(\nabla \bold{t}(\bold{q}^*) - \nabla \bold{t}(\bold{q}^\#))\bold{g}| &\le& |\lambda^{HDV}|   \|\bold{q}^*\|_2 \|\nabla^2 \bold{t}\|_{2,\infty} \|\bold{q}^* - \bold{q}^\#\|_2 \|\bold{g}\|_2, \\
|E| = |\lambda^{HDV} (\bold{q}^* - \bold{q}^\#) \nabla \bold{t}(\bold{q}^\#) \bold {g}| &\le& |\lambda^{HDV}| \|\bold{q}^* - \bold{q}^\#\|_2 \|\nabla \bold{t}\|_{2,\infty} \|\bold{g}\|_2, 
\end{eqnarray*}
where 
\begin{eqnarray*}
\|\nabla \bold{t}\|_{2,\infty} &=& \sup_{|\bold{q}| = q} \|\nabla \bold{t(\bold{q})}\|_2,  \\
\|\nabla^2 \bold{t}\|_{2,\infty} &=& \sup_{|\bold{q}| = q} \|\nabla^2 \bold{t(\bold{q})}\|_2,
\end{eqnarray*}
$\|v\|_2 = (\sum_i v_i^2)^{1/2}$ is the euclidean norm of a vector and $\|\nabla \bold{t(\bold{q})}\|_2$, $\|\nabla^2 \bold{t(\bold{q})}\|_2$ are the second (spectral) norms of operators $\nabla \bold{t(\bold{q})}$ and $\nabla^2 \bold{t(\bold{q})}$, respectively. Furthermore, $\rho(\nabla \bold{t})$ is the smallest eigenvalue of $\nabla \bold{t}(\bold{q})$ minimized over the compact set $|\bold{q}| = q$. Hence,
\begin{equation*}
L\rho(\nabla \bold{t}) \|\bold{g}\|_2 ^2 \le 
\left[(L\|\bold{f}^\#\|_2 + |\lambda^{HDV}|\|\bold{q}^*\|_2)\|\nabla^2 \bold{t}\|_{2,\infty} + (|\lambda^{CRV}| + |\lambda^{HDV}|)\|\nabla \bold{t}\|_{2,\infty}\right] \|\bold{q}^* - \bold{q}^\#\|_2 \|\bold{g}\|_2  
\end{equation*}
and, consequently,
\begin{equation}
\label{eq_LipC}
\|\bold{f^*} - \bold{f}^\#\|_2 \le \frac K {(\lambda^{CRV} - \lambda^{HDV})\rho(\nabla \bold{t})} \|\bold{q^*} - \bold{q^\#}\|_2
\end{equation}
for some constant $K$ dependent on $\|\nabla \bold{t}\|_{2,\infty}$, $\|\nabla^2 \bold{t}\|_{2,\infty}$, $q^{CRV}$, $q$ and number of routes $R$.

\end{proof}

\begin{remark}
The inverse operator may be continuous even though the forward operator is discontinuous, as is the case for the malicious strategy, see Example \ref{Ex_switching}. 
\end{remark}

Next, we demonstrate that if the route travel times are sums of link travel times given by BPR functions then $\bold{\nabla t}$ is essentially always positive definite. Let us begin with two definitions, which slightly generalize Definition  \ref{def_linkassignoper} to systems where link-additivity holds only on certain routes and not necessarily globally.

\begin{definition}
\label{def_linkadditive}
We say that a travel system is \emph{link additive} with respect to routes $1,\dots, R$ if the route delay functions from the travel time vector $\bold{t} = (t_1(\bold{q}), t_2(\bold{q})), \dots, t_R(\bold{q}))^T$ can be expressed as 
\begin{equation*}
t_r(\bold{q}) = \sum_{\alpha} I^{r\alpha} \tau_\alpha\left(\Lambda(\bold{q})_\alpha\right),
\end{equation*}
where $\tau_\alpha$ are continuously differentiable link-delay functions, i.e.  route travel times are sums of link travel times on links making up a route and every link travel time depends only on the total flow on this link and includes background traffic.\end{definition}

\begin{definition}
Routes $r_1,\dots, r_R$ in a link-additive system are \emph{linearly dependent} if there exists a vector $\bold{v} \neq \bold{0}$ in $\mathbb{R}^R$ such that  $\sum_{r=1}^R v_r r_r = 0$, which means that for every link $\alpha$ in the network $\sum_{r=1}^R v_r I^{r\alpha} = 0.$
\end{definition}
\begin{proposition}
\label{Prop_SLF}
In a link-additive system:
\begin{enumerate}
\item[i)] The objective function \eqref{eq_objective} can be expressed as 
\begin{equation}
\label{eq_objlink}
F(\bold{h},\bold{f}) = \sum_{\alpha} \left(\lambda^{HDV}h_\alpha + \lambda^{CRV}f_\alpha \right)\tau_\alpha(h_\alpha+f_\alpha)
\end{equation} 
where $h_\alpha = \Lambda(\bold{h})_\alpha, f_\alpha = \Lambda(\bold{f})_\alpha$ are the link flows corresponding to route flows $\bold{h}, \bold{f}$, respectively. 
 
\item[ii)] For a given route flow of all vehicles $\bold{q}$ if $\bold{q^{CRV}}, \bold{\tilde{q}^{CRV}}$ give rise to the same fleet link flows then the corresponding $\bold{q^{HDV}}, \bold{\tilde{q}^{HDV}}$ also result in the same HDV link flows

\item[iii)] If route assignments $\bold{q^{HDV}}$ and $\bold{\tilde{q}^{HDV}}$ result in the same link flows and $\bold{q^{CRV}}$, $\bold{\tilde{q}^{CRV}}$ result in the same link flows then $F(\bold{q^{HDV}}, \bold{q^{CRV}}) = F(\bold{\tilde{q}^{HDV}}, \bold{\tilde{q}^{CRV}})$, i.e. the objective function \eqref{eq_objective} is invariant to replacing HDV and CRV route flows by different route flows which result in the same link flows.  
\end{enumerate}
\end{proposition}
\begin{proof}
i) In a link-additive system we have
\begin{eqnarray*}
F(\bold{h},\bold{f}) &=& (\lambda^{HDV} \bold{h} + \lambda^{CRV} \bold{f}) \cdot \bold{t}(\bold{h} + \bold{f}) \\&=& \sum_{r=1}^R \left(\lambda^{HDV} {h_r} + \lambda^{CRV} {f_r}\right) t_r(\bold{h} + \bold{f})\\
&=& \sum_{r=1}^R \left(\lambda^{HDV} {h_r} + \lambda^{CRV} {f_r}\right) \sum_\alpha I^{r\alpha} \tau_\alpha(\Lambda(\bold{h} + \bold{f})_\alpha)\\
&=&  \sum_\alpha \left(\sum_{r=1}^R \left(\lambda^{HDV} {h_r} + \lambda^{CRV} {f_r}\right)I^{r\alpha}\right) \tau_\alpha(\Lambda(\bold{h} + \bold{f})_\alpha)  \\
&=& \sum_\alpha \left(\lambda^{HDV}h_\alpha + \lambda^{CRV}f_\alpha\right)\tau_\alpha(h_\alpha + f_\alpha)
\end{eqnarray*}

ii) is a simple consequence of the fact that $\bold{q^{HDV}} = \bold{q} - \bold{q^{CRV}}$ and linearity of the route-flow to link-flow conversion operator $\Lambda$, which implies $\Lambda(\bold{q^{HDV}}) = \Lambda(\bold{q}) - \Lambda(\bold{q^{CRV}})$. 

iii) is a direct consequence of i) as the objective function \eqref{eq_objective} can, in link-additive systems be expressed as a function dependent on link flows only, \eqref{eq_objlink}. 
\end{proof}

\begin{proposition}[Gradient of route travel times is positive semi-definite in link additive systems]
\label{Prop_grad}
Let $\bold{t} = (t_1(\bold{q}), t_2(\bold{q})), \dots, t_R(\bold{q}))^T$ be the travel times on routes $1,\dots, R$ in a link additive travel system. If $\tau_\alpha' > 0$ for every link $\alpha$, i.e. link delay functions are strictly increasing, then 
\begin{enumerate}
\item[i)] The gradient matrix $\nabla \bold{t}$ is positive semi-definite.
\item[ii)] The gradient matrix $\nabla \bold{t}$ is positive definite iff the routes $1, \dots R$ are linearly independent.
\end{enumerate}
\end{proposition}

\begin{proof}
i) Denoting $e_{ij}$ the $R \times R$ matrix with entry $(i,j)$ equal $1$ and all other entries equal $0$ we obtain
\begin{eqnarray*}
\nabla \bold{t} = \sum_{i=1}^R \sum_{j=1}^R  \frac {\partial t_i} {\partial q_j} e_{ij} &=& \sum_{i=1}^R \sum_{j=1}^R  \frac {\partial \sum_{\alpha} I^{r_i \alpha} \tau_\alpha(\Lambda(\bold{q})_\alpha)} {\partial q_j} e_{ij} \\ &=& \sum_{\alpha} \sum_{i=1}^R \sum_{j=1}^R I^{r_i \alpha} \frac {\partial  \tau_\alpha(\Lambda(\bold{q})_\alpha)} {\partial q_j} e_{ij} \\
&=& \sum_{\alpha} \sum_{i=1}^R \sum_{j=1}^R I^{r_i \alpha}I^{r_j \alpha} \tau_\alpha'\left(\Lambda(\bold{q})_\alpha)\right) e_{ij} \\
&=& \sum_{\alpha}\tau_\alpha'\left(\Lambda(\bold{q})_\alpha)\right) \left(\sum_{i=1}^R \sum_{j=1}^R I^{r_i \alpha}I^{r_j \alpha} e_{ij}\right) \\
&=& \sum_\alpha d_\alpha M_\alpha, 
\end{eqnarray*}
where we denoted $d_a := \tau_{\alpha}'(\Lambda(\bold{q})_\alpha)$ and $M_\alpha:= \sum_{i=1}^R \sum_{j=1}^R I^{r_i \alpha}I^{r_j \alpha} e_{ij}$.
Now, every $M_\alpha$ is symmetric and positive semidefinite with one eigenvalue equal $N_\alpha$ (number of routes containing $\alpha$) and $R-1$ null eigenvalues. Hence, $\nabla \bold{t}$ is positive semi-definite as a linear combination, with non-negative coefficients, of positive semi-definite matrices.

ii) Suppose first that routes $1, \dots, R$ are linearly dependent. Then there exists a vector $\bold{v} \neq \bold{0}$ in $\mathbb{R}^R$ such that   for every link $\alpha$ we have $\sum_{i=1}^R v_i I^{r_i \alpha} = 0$. Consequently $M^\alpha \bold{v} = \bold{0}$ for every link $\alpha$ and hence $\bold{v}$ is a null vector of $\bold{\nabla t}$ which implies $\bold{v}^T \bold{\nabla t}(\bold{q}) \bold{v} = 0$. 
On the contrary, if routes $1, \dots, R$ are linearly independent then for every $\bold{v} \neq \bold{0}$ in $\mathbb{R}^R$ there exists a link $\alpha_0$ such that $\sum_{i=1}^R v_i I^{r_i \alpha_0} \neq 0$. Consequently,  
\begin{eqnarray*}
\bold{v}^T \bold{\nabla t}(\bold{q}) \bold{v} =  \sum_\alpha d_\alpha \bold{v}^T M_\alpha \bold{v} = \sum_{\alpha \neq \alpha_0} d_\alpha \bold{v}^T M_\alpha \bold{v} + d_{\alpha_0} \bold{v}^T M_{\alpha_0} \bold{v}
\end{eqnarray*}
is positive as $\sum_{\alpha \neq \alpha_0} d_\alpha \bold{v}^T M_\alpha \bold{v}$ is non-negative and $\bold{v}^T M_{\alpha_0} \bold{v} = \left(\sum_{i=1}^R v_i I^{r_i\alpha_0}\right)^2 > 0$.

\end{proof}
As a corollary, we have the following generalization of Theorem \ref{th_inverse}, which holds in the case of linearly dependent routes. 

\begin{theorem}[Inverse of fleet assignment for linearly dependent routes]
\label{th_inverselindep}
Let $\bold{t} = (t_1(\bold{q}), t_2(\bold{q})), \dots, t_R(\bold{q}))^T$ be continuously differentiable delay functions (travel times) of routes $1,\dots, R$ in a link-additive system. 
If $\lambda^{HDV} <  \lambda^{CRV}$ then for every $\bold{q}$ such that the gradient matrix $\nabla \bold{t}(\bold{q})$ is positive definite $\bold{q^{CRV}} \in G(\bold{q - q^{CRV}})$ and $\bold{\tilde{q}^{CRV}} \in G(\bold{q - \tilde{q}^{CRV}})$ implies $\bold{q^{CRV}}, \bold{\tilde{q}^{CRV}}$ give rise to the same fleet link flows. Consequently, for $\lambda^{HDV} <  \lambda^{CRV}$ total route flow $\bold{q}$ (or equivalently total link flow) uniquely determines $CRV$ link flow optimizing the fleet objective.
\end{theorem}
\begin{proof}
Repeating the proof of Theorem \ref{th_inverse} we let $\bold{f^*}$, $\bold{\tilde{f}}$ to be two solutions of \eqref{eq_qCAV} resulting in different fleet link flows (note that if routes are linearly independent then different route flows \emph{always} result in different link flows). Then $\bold{g} = \bold{\tilde{f}} - \bold{f^*}$ is a route flow (with entries which may be negative) whose link flows (again with perhaps negative entries) do not all vanish. Consequently, as in the proof of Proposition \ref{Prop_grad}, there exists a link $\alpha_0$ such that $\sum_{i=1}^R g_i I^{r_i \alpha_0} \neq 0$ and hence
\begin{eqnarray*}
\bold{g}^T \bold{\nabla t}(\bold{q}) \bold{g} =  \sum_\alpha d_\alpha \bold{g}^T M_\alpha \bold{g} = \sum_{\alpha \neq \alpha_0} d_\alpha \bold{g}^T M_\alpha \bold{g} + d_{\alpha_0} \bold{g}^T M_{\alpha_0} \bold{g} > 0
\end{eqnarray*}
which contradicts \eqref{eq_last}. Hence, $\bold{q^{CRV}}, \bold{\tilde{q}^{CRV}}$ give rise to the same fleet link flows. We note that by Proposition \ref{Prop_SLF} different solutions with the same link flows imply the same value of the objective function.

\end{proof}
In the case of link-additive systems, in which \emph{every route} is available and optimized by CRVs, we can reformulate the inverse problem in terms of link-flows using \eqref{eq_objlink}, which results in an alternative, simplified proof of the inverse fleet assignment theorem. We state the result in a more general form, allowing the link flows to be weakly interdependent with proof which follows the same lines as the proof of Theorem \ref{th_inverse}. 

\begin{definition}[Dependent link-additive systems]
\label{def_deplinkadd}
Let $\bold{a} = (a_\alpha)$ be the vector of link flows in a travel system. We say that this travel system is \emph{dependent link-additive} if for every route $r$ from Origin to Destination we have
\begin{equation*}
t_r = \sum_{\alpha} I^{r\alpha} \tau_\alpha\left(\bold{a}\right),
\end{equation*}
where $\tau_\alpha$ are continuously differentiable link-delay functions. 
\end{definition}
\begin{remark}
Note that Definition \ref{def_deplinkadd} generalizes Definition \ref{def_linkadditive} as it assumes that link delay functions $\tau_\alpha$ depend on the whole vector of link flows $\bold{a}$ and not only on the link flow on the given link $\bold{a}_\alpha$. 
\end{remark}

\begin{theorem}[Inverse of fleet assignment for dependent link-additive systems]
\label{th_inverselinkadd}
Let $\boldsymbol{\tau}(\bold{a}) = (\tau_{\alpha})(\bold{a})^T$ be the vector of continuously differentiable link delay functions dependent on the vector $\bold{a} = (a_\alpha)^T$  of link flows. Let $q^{CRV}$ be the total flow of fleet vehicles (i.e. fleet size). If $\lambda^{HDV} <  \lambda^{CRV}$ then for any realisable (i.e. resulting from at least one route flow) link flow $\bold{a}$ such that the gradient matrix $\nabla \boldsymbol{\tau}(\bold{a})$ is positive definite there exists at most one realisable $\bold{a^{CRV}}$ such that $\bold{a^{CRV}} \le \bold{a}$ and 
$\bold{a^{CRV}} \in \Gamma(\bold{a} - \bold{a^{CRV}})$, where $\Gamma$ is the link assignment operator defined in Defintion \ref{def_linkassignoper}. 
\end{theorem}

\begin{proof}
The proof is a repetition of the proof of Theorem \ref{th_inverse} with the route flows replaced by link flows, i.e. $(\bold{q},\bold{q^{HDV}},\bold{q^{CRV}}) \to (\bold{a}, \bold{a^{HDV}}, \bold{a^{CRV}})$, the route delay functions $\bold{t}$ replaced by link delay functions $\boldsymbol{\tau}$ and the objective function $F$ replaced by the objective funtion $\Phi$. One only needs to verify that for  realisable link flows $\bold{\tilde{a}}$ and $\bold{a^*}$ the direction $\boldsymbol{\gamma}:= \bold{\tilde{a}} - \bold{a^*}$ is feasible for $\bold{a^*}$, which indeed follows by convexity of realisable link flows, Proposition \ref{Prop_convexity}.   
\end{proof}
\begin{corollary}
If the link delay functions are independent (e.g. given by the BPR functions) then $\boldsymbol{\tau}= (\tau_{\alpha}(a_\alpha))$ and hence
$\nabla \boldsymbol{\tau}$ is a diagonal matrix with entries $\tau'_\alpha$, which is clearly positive definite iff the derivatives $\tau'_\alpha(a_\alpha))$ are positive. 
\end{corollary}

\begin{remark}
The positive-definitedness of $\nabla \boldsymbol{\tau(a)}$ is a condition which is known to guarantee the existence and uniqueness of user equilibrium, see \cite{Smith1979, Dafermos, WatlingAsymmetric}. 
\end{remark}

\begin{remark}
The default setting is to consider all the routes between one OD pair. Then the CRVs 'compete' against HDVs from the same OD pair. However, we might also consider a restriction to only a subset of available routes (e.g. lorries using only selected routes). Then the HDVs having \emph{exactly the same} route choice set are considered as included in the objective function. All other vehicles (perhaps those with a larger choice set) are included as background traffic. 
\end{remark}

\begin{remark}
\label{Rem_mismatch}
Comparison of Fig. \ref{Fig_lambdas} and Fig.  \ref{Fig_convexity} reveals that there is a certain mismatch between the ease of optimization by the fleet and uniqueness of inverse. Namely, there exist cases (below we list only those with $\lambda^{CRV} \ge 0$) for which
\begin{itemize}
\item Fleet optimizes a convex objective and the inverse is unique (e.g. selfish),
\item Fleet optimizes a convex objective and the inverse is non-unique (e.g. social, altruistic),
\item Fleet optimizes a concave objective and the inverse is unique (e.g. malicious),
\item Fleet optimizes an objective which is neither convex nor concave on the whole domain and the inverse is unique (e.g. disruptive).
\end{itemize}
\end{remark}

\subsection{Generalization to multiple OD pairs}
\label{Sec_Gen}

\begin{definition}[Unit controlled by CRVs]
A basic \emph{unit of traffic controlled by CRVs} is given by $u_s = \{o_s, d_s, q^{CRV,s}, R_s\}$ where $s \in S$ is the index enumerating the units consisting of origin $o_s$, destination $d_s$, total flow $q^{CRV,s}$ of vehicles controlled by CRV operator and the set of routes $R_s$ available to every vehicle within the unit. 
\end{definition}
Note that two units can share the same origin and destination, however differ by the available sets of routes. 
\begin{definition} 
The flow controlled by the CRVs is given as 
$$U = \{u^s\}_{s=1}^{S}$$ where $S$ is the number of units and each $u_s$ is a unit controlled by CRVs.  
\end{definition}

In the case of multiple units controlled by CRVs, the total CRV-relevant flow is given by the concatenation of flows on units,
\begin{equation*}
\bold{q} = [\bold{q^{Tot, 1}}, \dots, \bold{q^{Tot, S}}]
\end{equation*}
where $\bold{q^{Tot, s}} = \bold{q^{HDV,s}} + \bold{q^{CRV,s}}$ is an $R_s$-dimensional vector of flow of vehicles in unit $s$, i.e. for 
\begin{eqnarray*}
\bold{q^{CRV,s}} = (q^{CRV,s}_1, \dots, q^{CRV,s}_{R_s})\\
\bold{q^{HDV,s}} = (q^{HDV,s}_1, \dots, q^{HDV,s}_{R_s})
\end{eqnarray*}
we have 
\begin{equation*}
\bold{q^{Tot,s}} = (q^{HDV,s}_1 + q^{CRV,s}_1, \dots, q^{HDV,s}_{R_s} + q^{CRV,s}_{R_s}) = ({q^{Tot,s}_1}, \dots {q^{Tot,s}_{R_s}}).
\end{equation*}
\noindent Therefore, $\bold{q}$ is a vector with $R = \sum_{s=1}^S R_s$ components. 

When considering the inverse fleet assignment theorem, we assume that the vector $\bold{q}$ is known and we denote $\bold{q^{CRV}} = [\bold{q^{CRV,1}}, \dots, \bold{q^{CRV,S}}]$ the concatenation of $CRV$ flows and $\bold{q^{HDV}} = [\bold{q^{HDV,1}}, \dots, \bold{q^{HDV,S}}]$ the concatenation of $HDV$ flows into $R$-vectors. Moreover, we assume that all the background traffic and its impact on travel times is implicitly included in the functions $t_1, \dots, t_R$.

\begin{theorem}[Inverse of fleet assignment for multiple OD pairs]
\label{th_generalize}
Let $\bold{t} = (t_1(\bold{q}), t_2(\bold{q})), \dots, t_R(\bold{q}))^T$ be the travel times on routes $1,\dots, R$, dependent on the total flow $\bold{q} = (q_1, \dots, q_R)$. 
Suppose $\bold{t}$ is continuously differentiable and $\bold{g}\cdot (\nabla \bold{t(q)}\bold{g})>0$ for every feasible direction $\bold{g} = (\bold{g^1}, \dots, \bold{g^S})\neq \bold{0}$ such that $|\bold{g^s}| = \bold{0}$ for every $s \in S$. Fix the flow of fleet vehicles, $q^{CRV,s}$ on every unit $s$ controlled by CRVs.  Then the mapping $\bold{q} \mapsto \bold{q^{CRV}}$ (defined for $\bold{q}$ such that $|q^{Tot,s}|\ge |q^{CRV,s}|$ for every $s$) is single-valued if $\lambda^{HDV} <  \lambda^{CRV}$. Equivalently, the fleet assignment operator $G$ is 1-1. 
\end{theorem}
\begin{proof}
The proof is exactly the same as the proof of Theorem \ref{th_inverse}. We only note that, as pointed out in Remark \ref{rem_feasibledir} it is enough to consider only feasible directions $\bold{g}$. In particular, Theorem \ref{th_generalize} holds true if $\nabla \bold{t(q)}$ is positive definite, i.e. $\bold{g}\cdot (\nabla \bold{t(q)}\bold{g})>0$ for \emph{every} direction $\bold{g} \neq \bold{0}$. 
\end{proof}

\begin{remark}
Theorems \ref{th_inverselindep} and \ref{th_inverselinkadd} can be generalized in a similar fashion to multiple OD pairs.
\end{remark}

\subsection{Inverse fleet assignment in the discrete case}

\label{Sec_discrete}

The discrete case in inverse fleet assignment is more difficult than the continuous one as the standard tools of differential calculus are unavailable. Moreover, as set out in Section \ref{Sec_realworld} the discrete fleet assignment according to a prescribed objective $(\lambda^{CRV}, \lambda^{HDV})$ is not straightforward. Assuming, however, as discussed in Section \ref{Sec_realworld} that the fleet controller's assignment consists in 1) identifying a local minimum (minima) 2) selecting a local minimum and 3) Arbitrary rounding-up flows to integer values, we have the following result.

\begin{proposition}
Assume the following pipeline:
\begin{enumerate}
\item Fleet controller knows the (integer) HDV flow $\bold{q^{HDV}}$
\item Fleet controller identifies a continuous local minimizer $\bold{q^{CRV}}$ of the objective function. 
\item Fleet controller rounds up the flows obtaining an integer CRV flow $\bold{q^{CRV,integer}}$ which is close to $\bold{q^{CRV}}$. 
\item City obtains the total integer flow $\bold{q}= \bold{q^{HDV}} + \bold{q^{CRV,integer}}$ yet it a priori knows neither $\bold{q^{HDV}}$ nor $\bold{q^{CRV,integer}}$.
\item City finds a flow $\bold{{q^{city}}}$ which lies in the image of the (continuous) fleet assignment operator $H$ such that $\|{\bold{q^{city}}} - \bold{q}\|$ is minimized. 
\item Using the inverse assignment theory (Theorem \ref{th_inverse} or Theorem \ref{th_inverselinkadd} or Theorem \ref{th_inverselindep}) city identifies the unique continuous HDV flow $\bold{q^{HDV,city}}$ such that $H(\bold{q^{HDV,city}}) = {\bold{q^{city}}}$.
\end{enumerate}
Then $\bold{q^{HDV,city}}$ is close to $\bold{q^{HDV}}$.
\end{proposition}
\begin{proof}
\begin{eqnarray*}
\|\bold{q^{HDV,city}} - \bold{q^{HDV}}\| &=& \|H^{-1}(\bold{q^{city}}) - H^{-1}(\bold{q^{HDV}} + \bold{q^{CRV}})\|\\ &\le& \mbox {Lip}(H^{-1}) \|\bold{q^{city}} - (\bold{q^{HDV}} + \bold{q^{CRV}})\| \\
&\le& \mbox {Lip}(H^{-1}) \left(\|\bold{q^{city}} - \bold{q}\| + \| \bold{q} -  (\bold{q^{HDV}} + \bold{q^{CRV}})\|\right)\\
&\le& 2\mbox {Lip}(H^{-1}) \|\bold{q^{CRV, integer}} - \bold{q^{CRV}}\|
\end{eqnarray*}
where we used the fact that 
$\|\bold{q^{city}} - \bold{q}\| = \min_{\bold{r}} \|\bold{r} - \bold{q}\| \le \|(\bold{q^{HDV}} + \bold{q^{CRV}}) - \bold{q}\|$ as $(\bold{q^{HDV}} + \bold{q^{CRV}})$ belongs to the image of $H$ and the Lipschitz constant $\mbox {Lip}(H^{-1})$ is bounded by Theorem \ref{Th_continuity}.
\end{proof}

\begin{example}[Nonuniqueness of the discrete inverse assignment]
Let there be two equivalent routes with total flow $\bold{q} = (50,50)$ and convex delay functions. Let the fleet apply the selfish strategy. For $q^{CRV} = 19$ we have the unique $\bold{q}^{CRV} = (9.5,9.5)$ for the continuous inverse assignment. However, in the discrete case there are two solutions $\bold{q}^{CRV} = (9,10)$ and $\bold{q}^{CRV} = (10,9)$.
\end{example}
\begin{example}
As above but for the malicious strategy. The malicious strategy (be it continuous or discrete) always assigns all the vehicles to one route as concave functionals are minimized in the corners of a polygonal domain. Therefore, for an integer $\bold{q}$ the discrete inverse and the continous inverse coincide. 
\end{example}

\subsection{Image of assignment operators under various strategies}

In this subsection we discuss the image of assignment  operators (discrete) given the fleet size and strategy in some trivial cases leaving a more detailed discussion to further reserach.

\begin{example}[No fleet vehicles]
If $q^{HDV} = q$ or, equivalently, $q^{CRV} = 0$ then the image of the assignment operator $H$ consists of all states $\bold{q}$ with $|\bold{q}| = q$ as $\bold{q^{HDV}}$ can be chosen arbitrarily. 
\end{example}

\begin{example}[Only fleet vehicles]
In this case the problem for various strategies degenerates and, as long as $\lambda^{CRV}>0$ it reduces to finding the system optimum. 
\end{example}

\begin{example}[Social strategy]
If $\bold{q} = \bold{q^{SO}}$ then for every $\bold{q^{HDV}}$ such that $\bold{q^{HDV}} \le \bold{q}$ we have $H(\bold{q^{HDV}}) = \bold{q^{SO}}$. Otherwise $H(\bold{q^{HDV}})$ is not equal to system optimum. 
\end{example}

\section{Conclusions, Open Problems and Future directions}
\label{Sec_Discussion}

\noindent Fleets of cooperating vehicles such as CAVs, once they start operating in our cities on a full scale might blend in with the surroundings and become indistinguishable from other vehicles. The collective routing of these vehicles, however, might increase human drivers' or overall travel times and cause congestion or chaos. 
Detection of cooperating fleet vehicles might prove rather difficult in the real world. 
Nevertheless, in this paper we showed that the first step of the process, i.e. identification of flows of collectively routing vehicles is theoretically achievable for fleet behaviours which are more selfish than altruistic -- a range which covers most standard fleet objectives. The future work \cite{JamrozDetectionInPrep} will encompass discussion of methods of detection of individual collectively routing vehicles and efficient algorithms necesary to achieve this goal. 

\,\\
\noindent Some other issues that we left open are: 
\begin{enumerate}
\item[i)] Proving the continuity of the inverse fleet assignment operator with respect to parameters. 
\item[ii)] Fully solving the discrete inverse problem. 
\item[iii)] Characterizing the image of the forward fleet assignment operator, i.e. which total assignments $\bold{q}$ can arise when the size of fleet is given by $q^{CRV}$?   
\item[iv)] Proposing efficient exact and heuristic algorithms of fleet assignment in nonstandard (non-convex) cases.  
\item[v)] Addressing the inverse fleet assignment problem with multiple fleets or fleets applying more elaborate objectives which are not linear combinations of HDV and CRV such as multiplicative objectives or multi-day objectives based on total payoff over the considered period, see also Appendix \ref{Sec_Stackelberg} for Stackelberg routing.
\item[vi)] Considering more general transportation settings such as varying departure times of agents, stochasticity and uncertainty of fleet assignment, more realistic delay functions or simulations in open source virtual environments such as SUMO \cite{SUMO} with fleets agents applying machine learning to optimize their routing, \cite{Akman2024, Akman2} etc. 
\end{enumerate}

\section*{Acknowledgement}
This work was financed by the European Union within the Horizon Europe Framework Programme (ERC Starting Grant COeXISTENCE no. 101075838). Views and opinions expressed are however those of the authors only and do not necessarily reflect those of the European Union or the European Research Council Executive Agency. Neither the European Union nor the granting authority can be held responsible for them.

\appendix

\section{Relation to Stackelberg and Nash routing} 
\label{Sec_Stackelberg}
In this section we compare the Stackelberg and Nash routing to the myopic routing used in this paper. In {\bf Stackelberg} routing, the atomic player (leader), being able to \emph{predict} the equilibrium behaviour of the non-atomic players (followers) given its own assignment, routes the vehicles so as to optimize its objective once the equilibrium is reached. More formally, we have the following definition, compare \cite{StackelbergRoughgarden, CorreaStackelberg}. 

\begin{definition}[Stackelberg equilibrium]
Stackelberg equilibrium \cite{Stackelberg} is a subgame perfect Nash equilibrium of the sequential game where the leader first commits to a given strategy and then the follower makes a choice \emph{knowing} the commitment of the leader.
\end{definition}

\noindent In our setting of mixed fleet -- human driver system a Stackelberg equilibrium occurs when 
\begin{enumerate}
\item[i)] the fleet operator, being able to predict the induced \emph{equilibrium} flows resulting from human driver assignment, chooses its own assignment such that its objective is optimized \emph{provided} the human drivers indeed assign themselves according to the respective equilibrium. 
\item[ii)] the human drivers assign themselves according to (Wardrop, Stochastic etc.) User Equilibrium, \emph{knowing} the assignment of the fleet.
\end{enumerate}
\noindent In {\bf Nash} routing, on the other hand, the fleet controller computes and applies every day the flows originating from a Nash equilibrium of the \emph{simultaneous} game, where neither the fleet nor the HDVs (treated as a group) can improve by unilaterally changing their assignment. To be able to apply this kind of routing, the fleet controller requires, similarly as for Stackelberg, an ability to predict the equilibrium flows of HDVs. 

\noindent In contrast, the {\bf myopically} routing fleet controller considered in this paper 
\begin{itemize}
\item First {predicts} the \emph{non-equilibrium} flows of HDVs on a given day based on behavioural models of HDVs and observation of previous HDV assignments. The way this prediction works does not play any role in this paper, however, as long as it is precise.
\item Then it selects an assignment which optimizes the one-day objective given the predicted HDV assignment. 
\item HDVs indeed behave according to the model used by the fleet controller. 
\end{itemize}

\subsection{Comparison of myopic and Stackelberg/Nash routing}
\noindent Comparing the myopic, Stackelberg and Nash routing we find that, Table \ref{Tab_routings},
\begin{itemize}
\item The myopic routing can be applied without precise knowledge of (often non-unique) equilibrium human behaviour and requires only a day-to-day human behaviour prediction model, which could potentially be very simple (e.g. use the flows from the day before in settings when the flows do not vary much day-to-day). On the other hand, for Stackelberg, an equilibrium human behaviour model is necessary to compute the HDV assignment for every possible assignment of the fleet. The same applies to Nash routing. 

\item Even if the optimal Stackelberg fleet routing can be computed and it is a pure strategy, there  may exist many (Wardrop, Stochastic) User Equilibria with the given fleet flow. This may happen even when there is no fleet at all when Wardrop equilibria are non-unique. 
If the fleet controller selects the Stackelberg flows and sticks to them, the HDV flows may converge to any of these subgame Nash equilibria. In this context we note that:
\begin{itemize}
\item The time-scale of this convergence is not clear \cite{IWH}.
\item The Stackelberg routing eventual efficiency is conditional on the \emph{initial human flows}. In contrast to standard Stackelberg games setting, where the follower(s) are supposed to \emph{choose} the equilibrium behaviour, in favour or not of the leader, leading to Strong/Weak or other Stackelberg equilibria, in our setting the followers only \emph{converge} to the equilibrium flow and have little control over which flow this convergence will end up with. 
\end{itemize}
\item Similarly, for Nash routing the eventual efficiency depends on whether HDV flows converge to the Nash equilibrium flows desired by the fleet operator. 

\item The myopic routing in simple settings often converges to a state where the flows of both the fleet and HDVs stabilize, which resembles a Nash Equilibrium (see \cite{JamrozSciRep}). Indeed, the fleet controller predicts that the Human Driver assignment will not change (much) the next day, and this indeed happens, and it routes the same flows the next day. This, however, is often distinct from Stackelberg routing which may be more efficient for the fleet controller, however is \emph{not} the best response to equilibrium HDV behaviour (see e.g. \cite{Yin}).

\item The myopic routing is (can be) always based on pure strategies as the fleet controller 'moves' second and there is no added value of applying a mixed strategy even when the payoffs of two or more strategies coincide. In contrast, for Stackelberg and Nash routing mixed strategies are often indispensable. 
 

\item The Stackelberg routing may not be more efficient than myopic routing even if the HDV flows eventually converge to the most favourable equilibrium, as this routing disregards the transient periods which the controller can still take advantage of, see Remark \ref{Rem_context}. The same applies to Nash routing. 

\item The optimal Stackelberg routing and myopic routing are optimization problems and so they always exist. On the other, under the reasonable practical assumption of only pure HDV strategies the Nash equilibrium and consequently Nash routing (even for mixed fleet strategies) may not exist, see Example \ref{Ex_MixedStackelberg}. 

\item It is NP-hard to compute exactly the Stackelberg Equilibrium and simple heuristic algorithms are not very efficient \cite{Jeroslow, StackelbergRoughgarden}.  Nevertheless for 'convex' players efficient algorithms for solving this bilevel problem based on solving variational inequalities are well-known \cite{Yang2007, VanVuren}, which however do not scale well, see \cite{LiBilevel} for recent approaches.  
\item If HDVs cannot be treated as a single player then Stackelberg leadership is not necessarily beneficial \cite{VonStengel}. 
\end{itemize}

Based on the above, we conclude that myopic routing, even though in many cases it is likely not optimal long-term,  is a simple to implement greedy strategy which maximizes the immediate reward and we consider it as a baseline for studying different behaviours of fleet controllers. Whether (and when) the Stackelberg routing is better on average than myopic routing is an open question. Finally, we note that if the fleet controller was to use the Stackelberg routing, which was known to the city, the problem of identifying fleet flows would cease to be an inverse problem and would amount to \emph{computing} the Stackelberg equilibrium and no total flows of vehicles would be needed, as is in the case of myopic routing setting considered in this paper. 

\begin{table}[h!]
\label{Tab_routings}
\caption{Comparison of different kinds of fleet routing}
\centering
\begin{tabular}{|c||c|c|c|}
\hline
Fleet routing type & {\bf Myopic} & {\bf Stackelberg} & {\bf Nash} \\ 
\Xhline{3\arrayrulewidth}
\makecell{Fleet assignment \\dynamics} & \makecell{Dynamic, \\ may $\rightarrow$ Nash Eq. }& \makecell{Static \\ (Equilibrium flows)} & \makecell{Static \\ (Eq. flows)}
\\
\hline
\makecell{Human model \\ required} & Day-to-day prediction & Equilibrium prediction & Equil. pred.
\\
\hline
Existence & Always& Always& Not always 
\\
\hline
Computability & Easy in Convex Cases& Hard& Hard 
\\
\hline
\makecell{Clarity which\\assignment to choose} & Yes & No & No 
\\
\hline
\makecell{Requires mixed\\strategies} & No & Yes & Yes 
\\
\Xhline{3\arrayrulewidth}
\makecell{HDV assignment\\dynamics}& \makecell{Dynamic, \\ may $\rightarrow$ Nash Eq.}& \makecell{Dynamic, \\ may $\rightarrow$ Stack. Eq.} & \makecell{Dynamic \\may $\rightarrow$ Nash Eq.}
\\
\hline
\makecell {Uniqueness when \\ stabilized} & No & No & No 
\\
\Xhline{3\arrayrulewidth}
\makecell{Inverse theory\\applicability} & always & \makecell{Trivializes \\ amounts to computing\\ Stackelberg Equilibrium } & \makecell{Only at \\ equilibrium when \\ Nash and myopic \\ coincide}
\\
\hline
\makecell{Usefulness of \\inverse theory} &  all dynamic settings & No & Limited\\
\hline
\end{tabular}
\end{table}


\subsection{Example: Stackelberg equilibria for malicious fleet controller}


In this section we show the differences between different kinds of routing on an example of malicious fleet controller. We note that equilibrium outcomes for standard 'convex' players, i.e. system optimal player (corresponding to social fleet in our setting), Cournot-Nash player (corresponding to selfish fleet in our setting) and  User Equlibrium player (corresponding to HDVs in our setting) are discussed in detail in \cite{Yang2007}, along with algorithms computing the respective equilibria, see also \cite{VanVurenVliet} for previous results in UE - System Optimal player setting.

\begin{example}[Optimal Stackelberg routing is mixed, Nash routing does not exist]
\label{Ex_MixedStackelberg}
Let the system consist of two independent routes with the same convex delay functions $t_1=t_2$, $q^{HDV}$ human drivers minimizing their travel time (i.e. tending to Wardrop User Equilibrium) and $q^{CRV}$ fleet vehicles. Suppose for simplicity that $q^{CRV}=q^{HDV}$. 
All Stackelberg equilibria are of the form
\begin{equation}
\label{eq_stackrouting}
\bold{q^{CRV}} = 
\begin{cases}
(q^{CRV}, 0) \mbox{ with probability } p, \\
(0, q^{CRV}) \mbox{ with probability } 1-p,
\end{cases}
\end{equation}
and $\bold{q^{HDV}}$ the corresponding unique induced user-equilibrated HDV routing, where $p \in (0,1)$. It may happen that the unique equilibrium is obtained for $p=0.5$
and $q = (q^{HDV}/2,q^{HDV}/2)$, or there can be two solutions. Consequently, optimal Stackelberg malicious fleet routing is stochastic and causes massive oscillations of travel times on routes. On the other hand, considering Nash equilibria, we note that any best response (i.e. in fact myopic) fleet routing puts all fleet vehicles on one route. This implies that a Nash equilibrium with pure strategy HDV routing (which we assume) does not exist as $(q^{CRV},0)$ induces best HDV response $(0,q^{HDV})$ which in turn induces best fleet response $(0,q^{CRV})$ and so on. Consequently, the fleet cannot use Nash routing as it does not exist.     
\end{example}
\noindent We first note that for $p=1$ the corresponding induced HDV assignment is $\bold{q^{HDV}} = (0,q^{HDV})$. In comparison, it can be shown that for $p=0.5$ the corresponding $\bold{q^{HDV}} = (q^{HDV}/2,q^{HDV}/2)$ incurs higher average travel times for HDVs by convexity of $t_1,t_2$. Therefore, \eqref{eq_stackrouting} with $p=1$ or $p=0$ (by symmetry) is not an optimal Stackelberg routing. Otherwise, the claim of Example \ref{Ex_MixedStackelberg} is a consequence of the following more general result.  
\begin{proposition}
\label{Prop_boundary}
For the malicious fleet of size $q^{CRV}$, aiming to maximize total HDVs travel time, and a system with two routes and HDVs minimizing their individual travel times, the optimal Stackelberg routing is supported on strategies $(q^{CRV},0)$ and $(0,q^{CRV})$.
\end{proposition}
\begin{proof}
The most general mixed strategy of the fleet is given by $\bold{q^{CRV}} = (\alpha q^{CRV}, (1-\alpha)q^{CRV})$ with probability $p(\alpha)$, where $p(\alpha)$ is a probability measure on $[0,1]$. Let $\bold{q^{HDV}} = (q^{HDV}_1, q^{HDV}_2)$ be the corresponding induced user-equilibrated HDV assignment, which means that if both routes are used 
\begin{equation}
\label{eq_ue1}
\int_{[0,1]} t_1\left(q_1^{HDV} + \alpha q^{CRV}\right) dp(\alpha) = \int_{[0,1]} t_2\left(q_2^{HDV} + (1-\alpha) q^{CRV}\right) dp(\alpha).
\end{equation}
\noindent We claim that there exists $k \in [0,1]$ such that for $p^k:= k\delta_0 + (1-k)\delta_1$ we have 
\begin{equation}
\label{eq_uepk}
\int_{[0,1]} t_1\left(q_1^{HDV} + \alpha q^{CRV}\right) dp^k(\alpha) = \int_{[0,1]} t_2\left(q_2^{HDV} + (1-\alpha) q^{CRV}\right) dp^k(\alpha),
\end{equation}
which is equivalent to 
\begin{equation}
\label{eq_ue2}
k t_1\left(q_1^{HDV}\right) + (1-k)t_1\left(q_1^{HDV} + q^{CRV}\right) = k t_2\left(q_2^{HDV} + q^{CRV}\right) + (1-k) t_2\left(q_2^{HDV}\right).
\end{equation} 
Indeed, for $k=0$ LHS of \eqref{eq_ue2} > LHS of \eqref{eq_ue1} and for $k=1$ LHS of \eqref{eq_ue2} < LHS of \eqref{eq_ue1}. Similarly, for $k=0$ RHS of \eqref{eq_ue2} < RHS of \eqref{eq_ue1} and for $k=1$ RHS of \eqref{eq_ue2} > RHS of \eqref{eq_ue1}. By continuity, the claim is proven. Let now $\tilde{t}_r$ be a modification of $t_r$, $r=1,2$ given by 
\begin{equation*}
\tilde{t}_r(q_r) := \begin{cases}
t_r\left(q_r^{HDV}\right) + \frac {q_r - q_r^{HDV}}{q^{CRV}}\left(t_r\left(q_r^{HDV}+q^{CRV}\right) - t_r\left(q_r^{HDV}\right)\right) & \mbox{ if } q_r \in (q_r^{HDV}, q_r^{HDV} + q^{CRV}),\\
t_r(q_r) & \mbox{ otherwise.}
\end{cases}
\end{equation*} 
Note that $\tilde{t}_r(q_r) > t_r(q_r)$ for every $q_r \in (q_r^{HDV}, q_r^{HDV} + q^{CRV})$ by strict convexity of $t_r$. Using this modification, we note that 
\begin{eqnarray*}
LHS\eqref{eq_ue1} &<& \int_{[0,1]} \tilde{t}_1\left(q_1^{HDV} + \alpha q^{CRV}\right) dp(\alpha) \\ &=& \int_{[0,1]} \left(\alpha t_1\left(q_1^{HDV} + q^{CRV}\right) + (1-\alpha)t_1\left(q_1^{HDV}\right)\right) dp(\alpha) \\ &=& \bar{\alpha} t_1\left(q_1^{HDV} + q^{CRV}\right) + (1-\bar{\alpha})t_1\left(q_1^{HDV}\right) =: F_1,
\end{eqnarray*}
where $\bar{\alpha} = \int_{[0,1]} \alpha dp(\alpha)$. Similarly,
\begin{equation*}
RHS\eqref{eq_ue1} < \bar{\alpha}t_2\left(q_2^{HDV}\right) + (1-\bar{\alpha})t_2\left(q_2^{HDV} + q^{CRV}\right) =: F_2.
\end{equation*}
Now, we claim that either $F_1 \le LHS\eqref{eq_ue2}$ or  $F_2 \le RHS\eqref{eq_ue2}$. Indeed, otherwise we obtain $1-k < \bar{\alpha}$ for the former, as both $F_1$ and $LHS\eqref{eq_ue2}$ are convex combinations of the same numbers and similarly $k < 1-\bar{\alpha}$ for the latter. This contradiction lets us conclude that either $LHS\eqref{eq_ue1} < LHS\eqref{eq_ue2}$ or $RHS\eqref{eq_ue1} < RHS\eqref{eq_ue2}$. Consequently, $RHS\eqref{eq_ue2} = LHS\eqref{eq_ue2} > RHS\eqref{eq_ue1} = LHS\eqref{eq_ue1}$, which proves that for any probability measure $p$ there exists $k$ such that the user equilibrium induced by measure $p^k$ involves longer (equal) travel times on the routes. Consequently, as the fleet objective is malicious, it suffices to consider only measures supported on $(q^{CRV},0)$ and $(0,q^{CRV})$ to find optimal Stackelberg routing. The case when the user-equilibrated induced flow $\bold{q^{HDV}}$ assigns all HDVs to only one route, for which \eqref{eq_ue1} becomes a strict inequality, e.g. "$<$", without loss of generality, can be reduced to the proof above by considering $\tilde{t}_1 = t_1 + \tau$, for some $\tau>0$ such that $\int_{[0,1]} \tilde{t}_1\left(q_1^{HDV} + \alpha q^{CRV}\right) dp(\alpha) = \int_{[0,1]} t_2\left(q_2^{HDV} + (1-\alpha) q^{CRV}\right) dp(\alpha)$.  
\end{proof}

\begin{remark}
\label{Rem_context}
In the context of Example \ref{Ex_MixedStackelberg} we note that: 
\begin{enumerate}
\item[i)] Any fixed pure strategy $\bold{q^{CRV}}=(q^{CRV}_1, q^{CRV}_2)$ used by the fleet controller would result in HDV equilibrium routing $\bold{q^{HDV}} = ((q^{HDV}+q^{CRV})/2, (q^{HDV}+q^{CRV})/2) - \bold{q^{CRV}}$, with fleet objective significantly worse than for the mixed strategy \eqref{eq_stackrouting} provided the delay functions are convex as assumed. 
\item[ii)] The myopic routing would always assign all fleet vehicles to the route with more HDVs, resulting in assignment $(q^{CRV}, 0)$ when ${q^{HDV}_1} > {q^{HDV}_2}$ or $(0,q^{CRV})$ when $q_1^{HDV} < {q_2^{HDV}}$, compare \cite{JamrozSciRep} Suppl. Fig. 9,10 for a qulitatively similar case, with two \emph{different} routes, however. Then, depending on the adaptation rate of HDVs, the number of HDVs on the route with CRVs would gradually decrease over several days until it would drop below the number of HDVs on the alternative route. Then CRVs would switch to the alternative route and the process would repeat. The average objective of the fleet over many days would then be better than the average fleet objective for optimal Stackelberg routing.

\item[iii)] If the HDVs were a group \emph{System Optimum} player instead of a User Equilibrium player, the game would turn into a zero-sum game of fleet vs. HDVs, for which the min-max theorem by von Neumann \cite{vonNeumann} ensures the same equilibrium whatever the order of players. Consequently, the Stackelberg and Nash equilibria would coincide. 
\end{enumerate}
\end{remark}

\begin{corollary}
It seems that the probabilistic Stackelberg routing in Example \ref{Ex_MixedStackelberg} leads to wildly varying travel times on routes. On the contrary, the myopic routing achieves a similar (or better) effect and does not imply complete chaos on the roads. Consequently, myopic routing could be the preferred option for fleet operators to prevent detection. 
\end{corollary}

\section{Proof of Proposition \ref{Prop_foil} iii}
\label{Sec_convconc}
Let us compute the second derivatives of $\Phi$ with respect to each $\alpha$. Note that the mixed derivatives vanish and hence $\Phi$ is convex iff $\partial^2_{\phi_\alpha \phi_\alpha} \Phi \ge 0$ for every $\alpha$.
\begin{eqnarray*}
\partial^2_{\phi_\alpha \phi_\alpha} \Phi = 2\lambda^{CRV} \tau' _\alpha (\eta_\alpha + \phi_\alpha) + (\lambda^{HDV}\eta_\alpha + \lambda^{CRV} \phi_\alpha) \tau'' _\alpha(\eta_\alpha + \phi_\alpha).
\end{eqnarray*}
For an {\bf affine} delay function $\tau_\alpha (q_\alpha) = 1 + q_\alpha$ we obtain
\begin{equation*}
\partial^2_{\phi_\alpha \phi_\alpha} \Phi = 2\lambda^{CRV},
\end{equation*}
and so $\Phi$ is convex as long as $\lambda^{CRV}>0$, i.e. for rational objectives. 

\noindent For a {\bf quadratic} delay function $\tau_\alpha (q_\alpha) = 1 + q_\alpha^2$ we obtain
\begin{equation*}
\partial^2_{\phi_\alpha \phi_\alpha} \Phi = 4\lambda^{CRV}(\eta_\alpha + \phi_\alpha) + 2(\lambda^{HDV} \eta_\alpha + \lambda^{CRV}\phi_\alpha)  = (4\lambda^{CRV} + 2\lambda^{HDV})\eta_{\alpha} + 6\lambda^{CRV} \phi_\alpha,
\end{equation*}
which is non-negative for any nonnegative flows $\eta_\alpha, \phi_\alpha$ iff $$\lambda^{CRV} \ge 0 \mbox{ and } \lambda^{HDV} \ge -2\lambda^{CRV}.$$ For instance, for $\lambda^{CRV} = 1$ and $\lambda^{HDV} = -1$ we obtain convexity, however for $\lambda^{CRV} = 0.25$ and $\lambda^{HDV} = -1$ there is no convexity over the whole domain.  
Note, however, that if we restrict ourselves to the feasible set $|\bold{q^{CRV}}| = q^{CRV}$ then e.g. for two routes we have $q^{CRV}_1 + q^{CRV}_2 = const$. The domain reduces to a line and considering $$\bar{\Phi}(\bold{h}, f_1) = \Phi(\bold{h}, (f_1, f-f_1)),$$ where $f$ is the total fleet flow, we obtain
\begin{eqnarray*}
\partial^2_{f_1,f_1} \bar{\Phi} = (4\lambda^{CRV} + 2\lambda^{HDV}) h + 6\lambda^{CRV} f
\end{eqnarray*} 
where $h = |\bold{h}|$. Consequently, convexity depends only on the relation of the total HDV and fleet flows for $\lambda^{HDV}<-2\lambda^{CRV}$ in the way that large enough total fleet flow makes the restriction of the objective function to the feasible set convex even though the objective function is not convex on the whole domain.  

For a {\bf general BPR} delay function 
$\tau_\alpha (q_\alpha) = \tau^0_\alpha \left(1 + d_\alpha \left(\frac {q_\alpha}{Q_\alpha}\right) ^{\gamma_\alpha} \right)$, $\gamma^\alpha \neq 1$, we obtain
\begin{eqnarray*}
C_\alpha \partial^2 _{\phi_\alpha \phi_\alpha} \Phi &=& 2\lambda^{CRV}  \gamma_\alpha (\eta_\alpha + \phi_\alpha)^{\gamma_\alpha - 1} + (\lambda^{HDV}\eta_\alpha + \lambda^{CRV} \phi_\alpha) \gamma_\alpha (\gamma_\alpha-1)(\eta_\alpha + \phi_\alpha)^{\gamma_\alpha -2} \\
&=& \gamma_\alpha (\eta_\alpha + \phi_\alpha)^{\gamma_\alpha -2} \left[  2\lambda^{CRV} (\eta_\alpha + \phi_\alpha) + (\lambda^{HDV}\eta_\alpha + \lambda^{CRV} \phi_\alpha) (\gamma_\alpha-1)\right]  \\
&=& \gamma_\alpha (\eta_\alpha + \phi_\alpha)^{\gamma_\alpha -2} \left[  \left(2\lambda^{CRV}+(\gamma_\alpha-1)\lambda^{HDV}\right)\eta_\alpha + \lambda^{CRV}\left(1 + \gamma_\alpha\right)\phi_\alpha  \right],
\end{eqnarray*}
where we denoted $C_\alpha := \tau_\alpha ^0 d_\alpha / Q^{\gamma_\alpha}_\alpha$ and the derivation is valid for $\gamma_{\alpha} > 1$. 
Therefore, $\Phi$ is convex if for every $\alpha$
\begin{equation}
\label{eq_convlambda}
\lambda^{CRV} > \frac {(1-\gamma_\alpha)}{2} \lambda^{HDV}.
\end{equation}
For instance, if all the links in the network have $\gamma_\alpha = 4$ then $\lambda^{CRV} > -0.75 \lambda^{HDV}$ guarantees convexity of $\Phi$. 
In general $\Phi$ is locally convex at $(\boldsymbol{\eta},\boldsymbol{\phi})$ if for every $\alpha$
\begin{equation*}
\phi_\alpha > \frac{2\lambda^{CRV}+(\gamma_\alpha-1)\lambda^{HDV}}{\lambda^{CRV}\left(1 + \gamma_\alpha\right)}   \eta_\alpha.
\end{equation*}
If \eqref{eq_convlambda} holds then the above inequality reduces to $\phi_\alpha > 0$. 

\section{Examples}
\label{Sec_Examples}
In this section we discuss examples of networks in which Theorem \ref{th_inverse} can and cannot be applied.
Across these examples, unless stated otherwise, we assume that the travel time via a route consisting of several links is a sum of travel times on these links, which in turn are given by increasing convex BPR functions, i.e. travel time on a section depends only on the flow on this section. This approach, although not fully realistic, is typical in transportation research, see e.g. \cite{Stabler}, and allows us to apply Proposition \ref{Prop_grad} to conclude that the assumptions of Theorem \ref{th_inverse} are satisfied provided the routes are linearly independent. 

\subsection{Two routes}
Here, we assume that the fleet vehicles are limited to two routes only, however these routes may be part of a larger urban network. The optimization according to \eqref{eq_obj} or \eqref{eq_objective} includes only the HDVs whose choice set are these routes, with all the remaining traffic treated as a fixed background. The most general two-route configuration is shown in Fig. \ref{Fig_2r}. The available routes are: 
\begin{itemize}
\item Route 1: c-a-d
\item Route 2: c-b-d
\end{itemize}
\begin{figure*}[h]
\centering
\includegraphics[scale=1.0]{"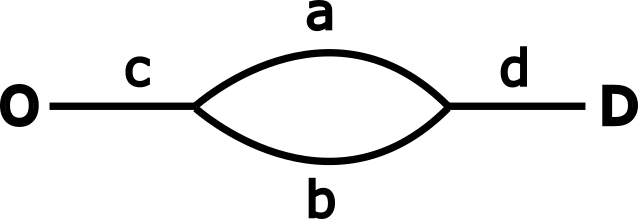"}
\caption{A two-route network.}
\label{Fig_2r}
\end{figure*}
In this setting, the first section (perhaps void) of the routes (c) is common, then the routes split (a or b), then join together and have a common final section (d), which may be void. Let $q_1$ be the flow on route 1 and $q_2$ the flow on route $2$.  
Then 
$$\bold{t}(\bold{q}) = \begin{bmatrix} t_1(q_1,q_2) \\ t_2(q_1,q_2)\end{bmatrix} =  (t_c(q_1 + q_2) + t_d(q_1 + q_2))\begin{bmatrix} 1 \\ 1\end{bmatrix} + \begin{bmatrix} t_a(q_1) \\ t_b(q_2) \end{bmatrix}$$ 
and
$$\nabla \begin{bmatrix} t_a(q_1) \\ t_b(q_2) \end{bmatrix} = 
t_a'(q_1) \begin{bmatrix}
1&0  \\
0&0 
\end{bmatrix}
+
t_b'(q_2)\begin{bmatrix}
 0&0  \\
 0&1
\end{bmatrix} 
+ 
\left(t_c'(q_1+q_2) + t_d'(q_1+q_2)
\right)\begin{bmatrix}
1 & 1 \\
1 & 1
\end{bmatrix} 
$$
where $t_a, t_b, t_c, t_d$ include the flow of other vehicles, if any, as a background flow. As routes $1$ and $2$ are linearly independent we conclude that Theorem \ref{th_inverse} holds in any system with choice between two-routes.

\subsection{Network '8'}
\label{Sec_Net8}
\begin{figure*}[h]
\centering
\includegraphics[scale=1.0]{"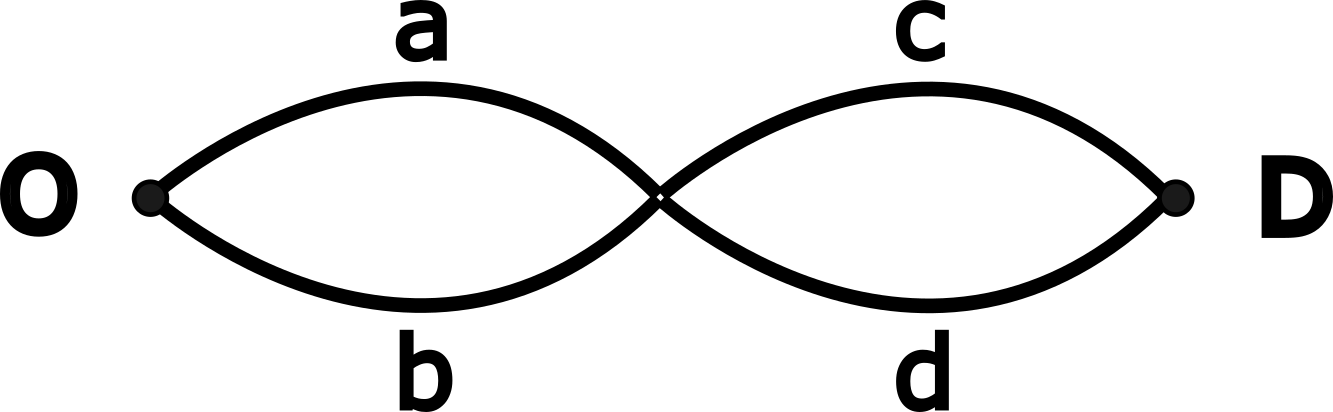"}
\caption{The '8' network with four available routes: 1: a-c, 2: a-d, 3: b-c, 4: b-d.}
\label{Fig_network8}
\end{figure*}
\noindent Network '8', the prototypical topology of dependent routes , is depicted in Fig. \ref{Fig_network8}. There are four routes, $r_1$: a-c, $r_2$: a-d, $r_3$: b-c, $r_4$: b-d, which are linearly dependent as $r_1 - r_2 - r_3 + r_4 = 0$. Indeed, increasing traffic on route $1$ by a given amount can be offset by decreasing the traffic on routes $2$ and $3$ by the same amount and increasing the traffic on route $4$ by this amount. In other words, route flow $(q_1 + \epsilon, q_2 - \epsilon, q_3 - \epsilon, q_4 + \epsilon)$ results in the same link flows as $(q_1,q_2,q_3,q_4)$ which clearly makes the inverse problem non-unique. This can be seen also by inspection of matrix $\nabla \bold{t}$, as

\begin{equation*}
\bold{t}(\bold{q}) = \begin{bmatrix}
t_a(q_1 + q_2) + t_c(q_1 + q_3)\\
t_a(q_1 + q_2) + t_d(q_2 + q_4)\\
t_b(q_3 + q_4) +  t_c(q_1 + q_3)\\
t_b(q_3 + q_4) +  t_d(q_2 + q_4)
\end{bmatrix}
\end{equation*}
\begin{eqnarray*} 
\nabla \bold{t}(\bold{q}) = 
\begin{bmatrix}
t_a' & t_a' &  &\\
t_a' & t_a' & &\\
 & & & \\
& & & 
\end{bmatrix}
+ 
\begin{bmatrix}
&  &  &\\
 &  & &\\
 & & t_b' & t_b'\\
& & t_b'&t_b' 
\end{bmatrix}
+ 
\begin{bmatrix}
t_c'&  &t_c'  &\\
 &  & &\\
t_c' & &t_c'  & \\
& & & 
\end{bmatrix}
+
\begin{bmatrix}
&  &  &\\
 & t_d' & &t_d'\\
 & & & \\
& t_d'& &t_d' 
\end{bmatrix}
=
\begin{bmatrix}
t_a' + t_c' & t_a' & t_c' &\\
t_a' & t_a' + t_d' & &t_d'\\
t_c' & & t_b' + t_c' & t_b'\\
&t_d' &t_b'& t_b' + t_d'
\end{bmatrix}
\end{eqnarray*}
and the sum of rows $1$ and $4$ of $\nabla \bold{t}$ is equal to the sum of rows $2$ and $3$. Nevertheless, Theorem \ref{th_inverselindep} ensures that CRV link flows are unique. Once the unique flows are known one can obtain the set of corresponding route flows which may consist of a single route flow or multiple route flows. To see this let us consider total OD flow $400$ and CRV flow $q^{CRV} = 100$ and equal convex link delay functions: $\tau_a = \tau_b = \tau_c = \tau_d$.  
\begin{itemize}
\item[a)] For $\bold{q} = (100,100,100,100)$ where each route receives the flow of $100$ vehicles leading to link flows $200$ on each link the unique link flow of selfish CRVs is $(a^{CRV}_a, a^{CRV}_b, a^{CRV}_c, a^{CRV}_d) = (50, 50, 50, 50)$. This link flow can result from different route flows $\bold{q}^{CRV}$ such as $(50, 0, 0, 50)$ or $(0,50,50,0)$ or $(25,25,25,25)$. 
\item[b)] For $\bold{q} = (200, 0, 0, 200)$ the resulting link flow is also $200$ on each link. The corresponding CRV link flow for selfish fleet is exactly the same as in a) however there exists only one route flow generating this link flow, i.e. $\bold{q}^{CRV} = (50,0,0,50)$.  
\end{itemize}

To conclude we note that if most routes are available then they are usually linearly dependent.  
\begin{proposition}
\label{Prop_exceeds}
If the number of routes exceeds the number of links, the routes are linearly dependent. 
\end{proposition}
\begin{proof}
Simple analysis of dimensions.
\end{proof}

\subsection{Dependent link flows, example of network to which the inverse theory does not apply, and signalized intersections}
In more realistic models of traffic networks, the link flows are interdependent. Consider for instance again the network depicted in Fig. \ref{Fig_2r}. While travel times via $c$ and $d$ may perhaps be independent of all other link flows, one may expect $\tau_a$ and $\tau_b$ to depend on both the flows on $a$ and $b$.  For instance, if
\begin{eqnarray*}
\tau_a(\bold{a}) &=& (1 + a_a + \delta_1 a_b)\\
\tau_b(\bold{a}) &=& (1 + a_b + \delta_2 a_a)\\
\tau_c(\bold{a}) &=& (1 + a_c)\\
\tau_d(\bold{a}) &=& (1 + a_d)
\end{eqnarray*}
the link delay function $\boldsymbol{\tau}$ has the gradient given by 
\begin{equation*}
\begin{bmatrix}
1 & \delta_1 && \\
\delta_2 & 1 &&\\
&&1&\\
&&&1
\end{bmatrix}.
\end{equation*}
Noting that the link delay gradient matrix does not have to be symmetric, to verify positive definitedness of the matrix computed above in the sense of Definition \ref{def_posdef} it suffices to verify that its symmetric part, 
\begin{equation*}
\begin{bmatrix}
1 & \frac{\delta_1+\delta_2}{2} && \\
\frac {\delta_1+\delta_2}{2} & 1 &&\\
&&1&\\
&&&1
\end{bmatrix}
\end{equation*}
is positive definite which is the case iff $\delta_1 + \delta_2 <2$. Consequently, if e.g. $\delta_1 = 2, \delta_2 = 1$ then the inverse theory does not hold. The corresponding route-flow formulation results in
$$\bold{t}(\bold{q}) = \begin{bmatrix} t_1(q_1,q_2) \\ t_2(q_1,q_2)\end{bmatrix} =  (\tau_c(q_1 + q_2) + \tau_d(q_1 + q_2))\begin{bmatrix} 1 \\ 1\end{bmatrix} + \begin{bmatrix} \tau_a(q_1,q_2) \\ \tau_b(q_1,q_2) \end{bmatrix}$$ 
and 
$$\nabla \bold{t(q)} = 
\begin{bmatrix}
1&\delta_1  \\
\delta_2 &1 
\end{bmatrix}
+ 
\left(\tau_c'(q_1+q_2) + \tau_d'(q_1+q_2)
\right)\begin{bmatrix}
1 & 1 \\
1 & 1
\end{bmatrix}.
$$
For the only feasible direction $\bold{g} = [1, -1]^T$ the second term cancels out and we obtain
\begin{equation*}
\bold{g} \cdot \nabla \bold{t(q)}\bold{g} = 2 - \delta_1 - \delta_2,
\end{equation*} 
which again is positive if and only if $\delta_1 + \delta_2 <2$. 

More examples with dependent link flows and delay functions $\bold{t}$ that are not strictly monotone (or, equivalently, $\nabla \bold{t}$ is not positive definite) can be found in \cite{BoylesBook}. 

The final example is the classical signalized intersection model due to Webster and Cobbe \cite{Webster} for which, in a simplified rough version the delay $d(x)$ is given by (for flows not exceeding the saturation flow)
\begin{equation*}
d(x) = \frac {9} {10} \left\{\frac {c(1-\lambda)^2}{2(1-\lambda x)} + \frac {x^2}{2\lambda s x(1-x)} \right\}
\end{equation*}
for some constants $\lambda, s, c$, which is strictly increasing (and even convex, which is not necessary here) as a function of flow $x$. By independence of link delay functions, the inverse theory holds in this case.

\subsection{Two flow units controlled by CRVs}
\begin{figure*}[h]
\centering
\includegraphics[scale=1.0]{"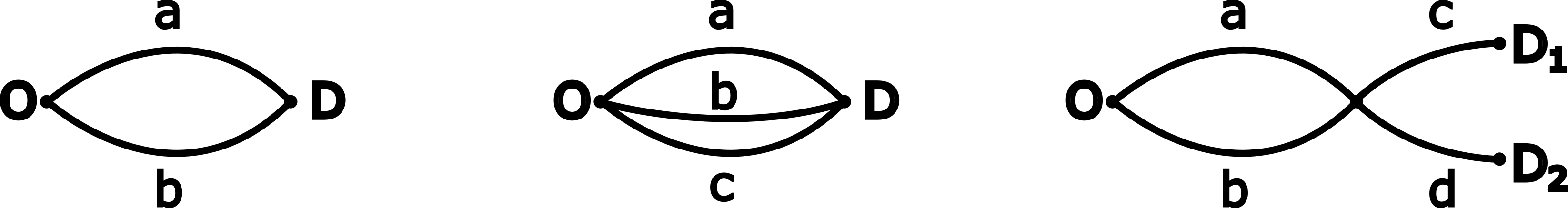"}
\caption{Left: A two route network, where link a cannot be used by heavy vehicles. Middle: a three-route network. Right: A network with two OD pairs.}
\label{Fig_2rs}
\end{figure*}
\noindent Consider the simple two-route network in Fig. \ref{Fig_2rs} left. Suppose there are two units of flow:
$u^1 = \{O, D, q^1, \{r^1_1,r^1_2\}\}$ and $u_2 = \{O, D, q^2, \{r^2_1\}\},$ where $r^1_1: a, r^1_2: b, r^2_1:b$. Unit $u^1$ may correspond to light passenger cars and $u^2$ to heavier/higher vehicles, which are not allowed to travel via $a$. Then, $\bold{q} = [\bold{q^1}, \bold{q^2}] = ({q^1_1}, {q^1_2}, {q^2_1})$,  $\bold{t} = (t_a(q^1_1), t_b(q^1_2 + q^2_1), t_b(q^1_2 + q^2_1))$ and 
\begin{equation*}
\nabla \bold{t} = \begin{bmatrix} t_a' & & \\ &t_b' &t_b' \\  &t_b'&t_b' \end{bmatrix},
\end{equation*}
which even though not positive definite, it is positive definite when restricted to feasible directions $\bold{g}$ which in this case are $\bold{g} = (\gamma, -\gamma, 0)$ as unit $2$ does not admit any variation of flow as there is only one  route available. Therefore, the inverse fleet assignment theorem, Theorem \ref{th_generalize} can be applied. 

Consider now the network in Fig. \ref{Fig_2rs}middle as before, however unit $2$ is allowed to use an additional route via $c$. A similar calculation as above yields 
\begin{equation*}
\nabla \bold{t} = \begin{bmatrix} t_a' & & &\\ &t_b' &t_b' &\\  &t_b'&t_b'& \\ &&&t_c' \end{bmatrix},
\end{equation*}
with feasible directions $\bold{g} = (\gamma^1, -\gamma^1, \gamma^2, -\gamma^2)$ which again yields $\bold{g} \cdot \bold{\nabla t} \bold{g} > 0$ for $\bold{g} \neq \bold{0}$ and satisfies the assumptions of Theorem \ref{th_generalize}.

Finally, consider an example of two OD pairs, depicted in Fig. \ref{Fig_2rs}right. 
There are two units of flow:
$u^1 = \{O, D_1, q^1, \{r^1_1,r^1_2\}\}$ and $u_2 = \{O, D_2, q^2, \{r^2_1, r^2_2\}\},$ where $r^1_1: ac, r^1_2: bc, r^2_1:ad, r^2_2:bd$. Then, $\bold{q} = [\bold{q^1}, \bold{q^2}] = ({q^1_1}, {q^1_2}, {q^2_1}, {q^2_2})$,  $\bold{t} = (t_a + t_c, t_b + t_c, t_a+t_d, t_b+t_d)$,
where $t_a = t_a(q^1_1 + q^2_1)$, $t_b = t_b(q^1_2 + q^2_2)$, $t_c=t_c(q^1_1 + q^1_2)$, $t_d = t_d(q^1_2 + q^2_2)$ and 
\begin{equation*}
\bold{\nabla t} = \begin{bmatrix}
t_a' + t_c' & t_a' & t_c' & \\
t_a' & t_a' + t_d' & & t_d' \\
t_c' & & t_b' + t_c' & t_b' \\
& t_d' & t_b' & t_b' + t_d'
\end{bmatrix}
\end{equation*}
which is identical to the matrix obtained in Section \ref{Sec_Net8} with dependent routes. In contrast to Section \ref{Sec_Net8}, however, where the routes were linearly dependent, here they are not. Indeed, a simple calculation reveals, that for a general feasible direction $\bold{g} = (\gamma^1, - \gamma^1, \gamma^2, -\gamma^2)$, we have $\bold{g} \cdot \bold{\nabla t} \bold{g} = (\gamma^1 + \gamma^2)^2 (t_c' + t_d')$ which is positive for $\bold{g} \neq 0$ and falls into the framework of Theorem \ref{th_generalize}.

\end{document}